\documentclass[a4paper,reqno]{amsart}
\pdfoutput=1

\usepackage{amsmath,amsfonts,amssymb,amsthm,stmaryrd,bbm,graphicx,mathtools,enumerate}
\usepackage{mathrsfs}
\usepackage[T1]{fontenc} 
\usepackage[latin1]{inputenc}
\usepackage[english]{babel}
\usepackage[tracking,spacing,kerning,babel]{microtype}
\usepackage{wasysym} 
\usepackage{color}
\usepackage{xcolor}
\usepackage{framed} 
\usepackage{wrapfig} 

\usepackage[final]{hyperref}   
\hypersetup{
	linktoc=page,
	linkcolor=red,          
	citecolor=blue,        
	filecolor=blue,      
	urlcolor=cyan,
	colorlinks=true           
}

\usepackage{lipsum} 

\usepackage{minitoc}
\usepackage{multicol}

\newtheoremstyle{mine}
{\baselineskip}
{\baselineskip}
{\itshape}
{
}
{\bfseries}
{.}
{.5em}
{#1 #2\ifx#3\relax\else~(#3)\fi}

\theoremstyle{mine}

\newtheorem{theorem}{Theorem}
\numberwithin{theorem}{section}
\newtheorem{corollary}[theorem]{Corollary}
\newtheorem{proposition}[theorem]{Proposition}
\newtheorem{lemma}[theorem]{Lemma}
\newtheorem{claim}[theorem]{Claim}
\newtheorem{definition}[theorem]{Definition}
 
\numberwithin{equation}{section}
\newtheorem{remark}[theorem]{Remark}

\colorlet{shadecolor}{blue!10}

\newcommand{\margin}[1]{\textcolor{magenta}{*}\marginpar[\textcolor{magenta} {  \raggedleft  \footnotesize  #1 }  ]{ \textcolor{magenta} { \raggedright  \footnotesize  #1 }  }}

\let\qed=\QED

\renewcommand{\epsilon}{\varepsilon}

\newcommand{\R}{\mathbb{R}}

\newcommand{\Z}{\mathbb{Z}}
\newcommand{\N}{\mathbb{N}}
\newcommand{\G}{\mathcal{G}}
\renewcommand{\S}{\mathbb{S}}

\newcommand{\1}{\mathbf 1}
\renewcommand{\d}{\mathbf d}
\newcommand{\avelio}[1]{#1}
\newcommand{\proj}[2]{\pi_{#1 \shortrightarrow #2}}
\newcommand{\tproj}[2]{\tilde \pi_{#1 \shortrightarrow #2}}
\newcommand{\projc}[2]{\mathring \pi_{#1 \shortrightarrow #2}}

\newcommand{\mm}{m}
\newcommand{\qq}{q} 
\newcommand{\cons}{C_{1\shortrightarrow 2}}

\def\calL{\mathcal{L}}

\def\var{\operatorname{Var}}

\def\P{\mathbb{P}} 
\def\E{\mathbb{E}} 
\def\md{\mid}

\def \eps {\epsilon}

\def\Bb#1#2{{\def\md{\bigm| }#1\bigl[#2\bigr]}}

\def\Eb{\Bb\E}

\def\FK#1#2#3{{\def\md{\bigm| } \P_{#1}^{\,#2}  \bigl[  #3 \bigr]}}
\def\EFK#1#2#3{{\def\md{\bigm| } \E_{#1}^{\,#2}  \bigl[  #3 \bigr]}}

\def \p {{\partial}}

\def\<#1{\langle #1\rangle}
\newcommand{\red}[1]{{\color{red}#1}}

\newcommand{\purple}[1]{{\color{purple}#1}}

\definecolor{darkgreen}{rgb}{0,0.6,0.05}
\newcommand{\dgreen}[1]{{\color{darkgreen}#1}}

\def\bi{\begin{itemize}}  
	\def\ei{\end{itemize}}
\def\bnum{\begin{enumerate}} 
	\def\enum{\end{enumerate}}
\def\bf{\bfseries}

\def\GFF{\mathrm{GFF}}

\def\GFF{\mathrm{GFF}}

\newcommand{\Vil}{\mathrm{Vil}}

\newcommand{\ccirc}{\mathbin{\mathchoice
		{\xcirc\scriptstyle}
		{\xcirc\scriptstyle}
		{\xcirc\scriptscriptstyle}
		{\xcirc\scriptscriptstyle}
}}
\newcommand{\xcirc}[1]{\vcenter{\hbox{$#1\circ$}}}

\title[]
{
Improved spin-wave estimate for Wilson loops in $U(1)$ lattice gauge theory}

\author{Christophe Garban}
\author{Avelio Sepúlveda}

\address
{Université Claude Bernard Lyon 1, CNRS UMR 5208, Institut Camille Jordan, 69622 Villeurbanne, France \, and Institut Universitaire de France (IUF)}
\email{garban@math.univ-lyon1.fr}

\address{Universidad de Chile,  Departamento de Ingeniería Matemática and Centro de Modelamiento Matemático (AFB170001), UMI-CNRS 2807, Beauchef 851, Santiago, Chile.}
\email{lsepulveda@dim.uchile.cl}

\begin{document}

\maketitle

\begin{abstract}
In this paper, 
we obtain bounds on the Wilson loop expectations in 4D $U(1)$ lattice gauge theory which quantify the effect of topological defects. In the case of a Villain interaction, by extending the non-perturbative technique introduced in \cite{GS2}, we obtain the following  estimate for a large loop $\gamma$ at low temperatures:
\[
|\<{W_\gamma}_{\beta}| \leq \exp \left(-\frac{C_{GFF}} {2\beta}(1+C \beta e^{- 2\pi^2 \beta} )(|\gamma|+o(|\gamma|)) \right)\,.
\]
Our result is in the line of recent works \cite{Sourav,cao,malin,forsstrom2021decay} which analyze the case where the gauge group is discrete. 
In the present case where the gauge group is continuous and Abelian, the fluctuations of the gauge field decouple into a Gaussian part, related to the so-called {\em free electromagnetic wave} \cite{gross1983convergence,driver1987convergence}, and a gas of {\em topological defects}. As such, our work gives new quantitative bounds on the fluctuations of the latter which complement the works by Guth and Fröhlich-Spencer \cite{guth1980,FSrestoration}.

Finally, we improve, also in a non-perturbative way, the correction term from $e^{-2\pi^2\beta}$ to $e^{-\pi^2\beta}$ in the case of the free-energy of the system. This provides a matching lower-bound with the prediction of Guth \cite{guth1980} based on renormalization group techniques.

\end{abstract}


\section{Introduction}

\subsection{Context.} 
$U(1)_4$ lattice Gauge theory is the statistical physics model on $\Z^4$ with $U(1)$ gauge symmetry which is relevant to the study of quantum-electrodynamics. In this paper (as in \cite{GS2}), we will focus on the Villain-version of the $U(1)$ lattice gauge theory and we will stick to the case of {\em pure} gauge theory (i.e. without coupled matter). 
It is defined as follows 
on a finite box $\Lambda \subset \Z^4$.
(See Definition \ref{d.Villain_final} for a more complete definition which includes the case of Dirichlet boundary conditions and Proposition \ref{pr.IVL} for its infinite-volume limit on $\Z^4$).


\begin{definition}[Villain $U(1)$-lattice gauge theory]\label{d.Villain}  Let $\Lambda\subseteq \Z^4$ be a finite box. Let $\vec E(\Lambda)$ denote the oriented edges of $\Lambda$ and $F(\Lambda)$ denote the faces (or plaquettes) of $\Lambda$. The Villain  $U(1)$-lattice gauge theory 
with free boundary conditions corresponds to the probability measure $\P_\beta^{Vil}$  on 
 \begin{align*}
 C^1_{\S^1}:=\{ \theta \in [-\pi, \pi)^{\vec E(\Lambda)}, \text{ s.t. } \theta(e)=-\theta(e^{-1}), \, \forall e \in \vec E(\Lambda)\}
 \end{align*} whose Radon-Nikodym derivative w.r.t the Lebesgue measure $d\theta$ on $C^1_{\S^1}$ is given by  
\begin{align}\label{e.Villain_mar}
\FK{\beta}{Vil}{d\theta} \propto  
 \prod_{f\in F(\Lambda)}  \sum_{m\in \Z} \exp\left( -\frac \beta 2  \left (2\pi m +\sum_{e\in f}\theta(e) \right )^2 \right)d\theta\,,
\end{align}
where for each $f\in F(\Lambda)$, one fixes an arbitrary orientation of $f$ and the sum $\sum_{e\in f}$ is over its corresponding oriented edges.
\end{definition} 
In this paper, the symbol $\propto$ stands for ``proportional to'' and is used throughout this text in order to avoid writing down the renormalization constant \[Z_{\beta,\Lambda}^{Vil}:=\int_{C^1_{\S^1}} \prod_{f\in F(\Lambda)}  \sum_{m\in \Z} \exp\left( -\frac \beta 2  \left (2\pi m +\sum_{e\in f}\theta(e) \right )^2 \right)d\theta,\] that makes $P_\beta^{Vil}$ is a probability measure.

We now introduce \textbf{Wilson loop observables} which are important gauge-invariant observables of this model and which may be defined as follows.
Let $R$ be a rectangle living in a two dimensional hyperplane parallel to the main axes, and let its boundary be represented by a closed oriented loop $\gamma$. 
The Wilson loop observable associated with $\gamma$ is given by
\begin{align}
W_R(\theta) = W_\gamma(\theta)=W_\gamma:= \prod_{e\in \gamma} e^{i \theta}.
\end{align}
This observable plays a key role as the gauge theory will be {\em confining} or not depending on its asymptotic decay as $R$ grows. (See \cite{guth1980,FSrestoration}).  It has been proved in \cite{guth1980,FSrestoration} that this model exhibits the following striking phase transition.

\begin{theorem}[Perimeter versus area law transition \cite{guth1980,FSrestoration}]\label{th.FS82}
When $\beta$ is large enough, there exists $c(\beta)>0$ so that 
\begin{align*}\label{}
|\EFK{\beta}{}{W_\gamma}| \geq \exp(-c(\beta) |\gamma|)\,,
\end{align*}
uniformly in rectangle loops $\gamma$. This is called the perimeter-law and corresponds to the deconfining phase. 

When $\beta$ is small enough, there exists $\tilde c(\beta)>0$ such that 
\begin{align*}\label{}
|\EFK{\beta}{}{W_\gamma}| \leq \exp(-\tilde c(\beta) \mathrm{Area}(\gamma))\,.
\end{align*}
This is called the area law and it corresponds to the confining phase. 
\end{theorem}

This phase transition shares some similarities with the BKT transition for the XY and Villain model in 2D (proved in \cite{FS}). Indeed, as we shall see below in the case of a Villain interaction, the gauge field  $\theta$ decouples into a Gaussian part and a {\em Coulomb-type} part which corresponds to {\em topological defects} (to be more precise, the decoupling will be proved to hold for the 2-form $\d \theta$ rather than for $\theta$ itself).

\subsection{Main result.}

The objective of this paper is to prove upper bounds on Wilson loop observables which quantify the effect of the topological defects at low temperatures. Note that at high temperature, topological defects are known to play a key role as they are fully responsible for the appearance of the area law/confining phase. 
Our main result may be stated as follows. (See also Theorem \ref{th.main_finite} for a more precise statement). 
\begin{theorem}\label{th.main}
Consider Villain $U(1)$-lattice gauge theory in a graph $\Lambda\subseteq \Z^4$ with either zero or free boundary conditions (also $\Lambda$ may be a finite cube or the infinite lattice).  There exists a constant $C>0$ such that for any $\beta\geq 1$ and any loop $\gamma$ which is sufficiently rectangular and sufficiently far from $\p \Lambda$, then the Wilson loop observable $W_\gamma$ satisfies 
\begin{align}\label{e.WL}
|\EFK{\beta}{}{W_\gamma}| \leq  \exp \left(-\frac{C_{GFF}} {2\beta}(1+ C \beta e^{-2\pi^2\beta} )(|\gamma|+o(|\gamma|)) \right)\,,
\end{align}
where $|\gamma|$ is the perimeter of the loop $\gamma$ and where the constant $C_{GFF}$ is defined out of $G_{\Z^4}$, the Green's function of the graph\footnote{By the Green's function of the graph of $\Z^4$, we mean here the Green's function of the simple random walk divided by the degree, i.e. $8$.} on the vertices of $\Z^4$ (see  Section \ref{ss.Laplacian_0}), as follows
\begin{align}\label{e.cGFF}
C_{GFF}:= \sum_{k\in \Z} G_{\Z^4}(0, k \,e_1)\,. 
\end{align}
\end{theorem}

Based on RG techniques and inspired by the  seminal work \cite{Kadanoff} on the $2d$ Villain model, Alan Guth predicted in \cite{guth1980} the following behaviour as $\beta\to \infty$ for large Wilson loop observables:
\begin{align}\label{e.AlanGuth}
\EFK{\beta}{}{W_\gamma} = \exp \left(-\frac{C_{GFF}} {2\beta}(1+  e^{-\pi^2 \beta +o(\beta)})(|\gamma|+o(|\gamma|) \right)\,.
\end{align}
Our results also allow us to improve our correction term to $e^{-\pi^2\beta}$ in the context of the free-energy. To be more precise, the following theorem upper bounds the derivative of the free-energy of a Villain $U(1)$-lattice gauge theory for either free or $0$-boundary condition

\begin{theorem}\label{t.d_free_energy}
	Take the graph $\Lambda_j:=[-j,j]^4\cap \Z^4$. Then, for any $\delta>0$ there exists a $\beta_0>0$ such that for any $\beta>\beta_0$
	\begin{align}\label{e.d_free_energy}
	\limsup_{j\to \infty}  \frac{1}{4(2j)^4}\left( \frac{\partial}{\partial \beta }\ln Z^{Vil}_{\beta,\Lambda_j}\right) \leq -\frac{3}{4}\left(\frac{1}{2\beta} + \frac{1}{2}e^{-\pi^2 (\beta+\delta)}\right),
	\end{align}
	for either free or zero boundary condition.

\end{theorem}

Let us discuss the terms appearing in \eqref{e.d_free_energy}. The term $4(2j)^4$ corresponds to the degrees of randomness that have functions in $E(\Lambda)$\footnote{More precisely $\binom{4}{1}(2j+1)^3(2j)$ for free boundary condition and $\binom{4}{1}(2j)(2j-1)^3$ for the $0$-boundary case.}. The term $3/4$ comes from the fact that the linear function $\theta \mapsto (\sum_{e\in f} \theta(e))_{f\in F(\Lambda)}$ has a non-zero kernel.  Then there are the two summands. The first is exactly the one coming from the Gaussian spin-wave. The second one is the most interesting one as it comes from the topological defects of vortices, and thus $e^{-(\pi^2+\delta)\beta}$ corresponds to the correction term predicted by Alan Guth \cite{guth1980}.


\subsection{Links with previous works.}\label{ss.links}
Let us briefly make some connections with other works in the subject. 

\bnum
\item The following lower bound (at low temperature only) may be easily extracted from the seminal work \cite{FSrestoration}: there exists $\beta \mapsto \eps(\beta)$ which goes to zero as $\beta\to \infty$ and which is such that when $\beta$ is large enough, 
\begin{align}\label{e.WL}
|\EFK{\beta}{}{W_\gamma}| \geq  \exp \left(-\frac{C_{GFF}} {2\beta}(1+\eps(\beta)) (|\gamma|+o(|\gamma|)) \right)\,.
\end{align}
As such our main result complements the results from \cite{guth1980,FSrestoration} and implies the following lower bound on Fröhlich-Spencer correction exponent 
\[
\eps(\beta) \geq C \beta \exp\left (-\frac {(2\pi)^2} 2 \beta\right )\,.
\]


\item  In the works \cite{gross1983convergence, driver1987convergence}), Gross and Driver show respectively that $U(1)_3$ and $U(1)_4$ lattice gauge theories (i.e resp on $\Z^3$ and $\Z^4$) rescale as the mesh of the lattice goes to zero to the {\em free electromagnetic wave} on $\R^3$ (resp. $\R^4$). As opposed to our present setting, topological defects do not play a role in Gross' result. This is due to the fact that in the case $d=3$, the natural scaling limit $a\Z^3 \to \R^3$ leads to a renormalized inverse temperature $\beta_a := \beta a^{-1}$ when the mesh $a\searrow 0$. At such low temperatures, vortices  do not play a visible role anymore. In $d=4$, the inverse temperature does not scale anymore with the mesh $a$ and the setup then corresponds to ours. Driver obtains a convergence towards the {\em free electromagnetic wave} with an effective inverse temperature $\beta_{eff}=\beta \alpha^{-1}$ (following notations from \cite{driver1987convergence}). Our present result thus implies quantitative  bounds on the correction term $\alpha$. More importantly, one cannot deduce improved spin-wave estimates for Wilson observables from \cite{driver1987convergence} as Wilson loop observables are too degenerate to be still measurable in the continuum limit.
 
\item There has been an intense activity recently on the analysis of Wilson loop observables for discrete gauge groups on $\Z^4$. It started with the work \cite{Sourav} for the gauge group $G=\Z_2$ followed by the works  \cite{malin,cao} which considered respectively finite Abelian groups and general discrete groups (see also the recent \cite{forsstrom2021decay}). There are two main differences with our present work: 
\bnum
\item In these works, as explained for example in \cite{cao}, the proofs require to focus on discrete gauge groups (in particular for the definition of vortices) while our present method allows us to deal with the gauge group $U(1)$. On the other hand, our proof technique would not extend to discrete gauge groups as we deeply rely on the {\em spin-wave} decoupling which to our knowledge does not have an analog in the discrete case. As such our work is complementary to \cite{Sourav,malin,cao}. See also the related Remark \ref{r.SouravUB}. 
\item A second main difference is that in \cite{Sourav,malin,cao}, the focus is in obtaining a precise evaluation of Wilson loop observables $\EFK{\beta}{}{W_\gamma}$ in the regime where the observable is bounded away from 0. For example when $G=\Z_2$, it is shown in \cite{Sourav} that as $\beta\to \infty$, rectangular-enough loops $\gamma$ that have length $|\gamma|\asymp e^{12\beta}$  satisfy $\EFK{\beta}{}{W_\gamma} = e^{-2 |\gamma| e^{-12\beta}} + o(1)$. In our present case we shall also extract from our proof precise estimates of $\EFK{\beta}{}{W_\gamma}$ in the regime where it is non degenerate (this happens for $U(1)$-lattice gauge theory for much shorter loops $\gamma$ of length $|\gamma| \asymp \beta$). Even though this is not the focus of this paper as topological defects are invisible at that scale, we included in Corollary \ref{c.sourav} a statement in the spirit of \cite{Sourav,malin,cao}.  

In this paper, we rather focus on establishing bounds on Wilson Loop observables which hold for arbitrary large macroscopic loops and which reveal the influence of vortices.
(See also Remark \ref{r.SouravUB}). 

\enum

\item Glimm-Jaffe obtained in \cite{glimm1977quark} an improvement w.r.t the perimeter decay for Wilson loops observables in $3d$ $U(1)$ lattice gauge theory. This implied the confinement of quarks in $3d$ $U(1)$ gauge theory at all temperatures. We obtain this result (for the Villain interaction) as a corollary of the spin-wave decoupling property (Proposition \ref{p.decouplingF}).

\item Using deep homogenization and PDE techniques, Dario and Wu obtained in \cite{DarioWu} the existence of  an effective temperature in the context of $3d$ Villain model. It is possible that their techniques would extend to the present setting of $4d$ lattice gauge theory (with Villain interaction). If so this would give the existence of an effective temperature $\beta\mapsto \beta_{eff}$ which would capture the effect of vortices for large macroscopic Wilson loops. Our present analysis would then provide lower bounds on the deviation of $\beta_{eff}$ from $\beta$. 

\item It would be interesting to try extending the results of this paper when matter is coupled to the gauge field. See for example \cite{wenzel2008percolation} as well as the recent paper \cite{forsstrom2021wilson} which deals with finite Abelian groups. 

\item The references we mentioned above deal with lattice gauge theories (and their scaling limits). Let us stress that there is also a vast literature on building gauge theories in the continuum, we refer to \cite{MR2006374} and references therein as well as to the recent works \cite{shen2021stochastic,chandra2020langevin} which apply new stochastic quantization ideas in order to build the $2D$ Yang-Mills  measure. 

\enum

\subsection{Idea of proof.}


The proof  follows closely the analysis in \cite{GS2}, however we shall see that the framework of lattice gauge theory presents challenges that do not exist in the case of the $2d$ Villain model. Here is a short outline of how the proof works. 


\begin{enumerate}
	\item First, we establish a  decoupling in the Villain $U(1)$-lattice gauge theory between a suitably defined {\em spin-wave} (which will turn out to be the gradient of a GFF on $1$-forms) and a cloud of {\em topological defects}, which behaves like a certain Coulomb gas defined on the 3-cells (see Proposition \ref{p.decouplingF}). 
Our analysis is based on discrete differential calculus  and shares some similarities with the arguments in \cite{FSrestoration}, 	except as in \cite{GS2}, the emphasis here is on the introduction of a new probabilistic object: the joint coupling $(\theta,m)$ where the $1$-form $\theta\sim \P_\beta^{Vil}$ and where $m$ is a random $2$-form whose quenched law given $\theta$ will be of great use. (N.B. both $\theta$ and $m$ appear in the definition of the Villain-interaction, the novelty from \cite{GS2} is to promote the role of the summation variable $m$ in the partition function of the Villain model to a proper random variable whose fluctuations can be computed efficiently). 
A significant difference with \cite{GS2} is that our decoupling would not hold at the level of $\theta$ (as it does for the $2d$ Villain model)  but only after applying the pushforward under $\d$ and considering the $2$-form $\d \theta$.

This more subtle decoupling statement allows us in Section \ref{ss.algo} to extend the local sampling algorithm for the Coulomb gas introduced in \cite{GS2} in dimension $n=2$  to any dimensions $n\geq 3$. 
	

	\item We then show that the cloud of topological defects give a contribution comparable to that of the spin-wave when they are tested again spread enough functions (Lemma \ref{l.Fourier_m}). Here lies the main difference between Theorem \ref{th.main} and Theorem \ref{t.d_free_energy}. The bounds for the first one is in a certain sense related to the worst case scenario for $m$, while the second one is related to its mean value.
	\item Finally, we show that the function which is associated with the Wilson loop observable is sufficiently well-spread for us to apply step (2). 
	This step also differs from \cite{GS2}. Indeed the specificity of the Wilson loop observable requires us to understand the behaviour of the inverse of the Laplacian on the edges of the graph instead of on the vertices. The study of this Laplacian is carried out in Section \ref{s.Laplacian_1} and the fact that the energy is indeed well-spread is obtained in Proposition \ref{p.spread_energy}.
\end{enumerate}

\subsection{Organization of the paper.} 
In Section \ref{s.Pr} we start with some preliminary background on discrete differential calculus on $\Z^d$ and we define the main statistical physics models used throughout.  In Section \ref{s.Laplacian_1}, we analyze the Green operator $\Delta^{-1}$ when acting on the $1$-forms of $\Lambda \subset \Z^4$. We shall focus on both free and Dirichlet boundary conditions.  The main purpose of Section \ref{s.OD} is to introduce the {\em gradient spin-wave}: a Gaussien field on the $2$-cells  which plays a key role in the decoupling  property proved in Section \ref{s.decoupling}. Section \ref{s.energy} then computes the Dirichlet energy of a Wilson loop $\gamma$. Section \ref{s.coro} gives several useful direct corollaries of the decoupling statement Proposition \ref{p.decouplingF}. Finally Section \ref{s.proof} concludes the proof of the main theorem \ref{th.main} and Section \ref{s.FE} proves the $e^{-\pi^2 \beta}$ correction to the free energy predicted in \cite{guth1980}.

\subsection{Acknowledgements.}
We wish to thank Malin Palö Forsström for useful discussions. 
The research of C.G. is supported by the ERC grant LiKo 676999 and the research of A.S was  supported by the ERC grant LiKo 676999 and is now supported by ANID/PIA Apoyo a Centros Científicos y Tecnológicos de Excelencia AFB 170001 and FONDECYT iniciación de investigación N° 11200085.

\section{Preliminaries}\label{s.Pr}

\subsection{Integer-valued Gaussian random variable.}\label{ss.IVG}
{We follow closely the presentation in \cite{GS2} to which we refer for more details. \label{ss.IG}Integer-valued Gaussian random variable, sometimes called discrete Gaussian variables, are normal random variables conditioned to take values in $\Z$. More precisely, we define $X\sim \mathcal N^{IG}(a,\beta)$ if a.s. $X\in \Z$ and for any $k\in \Z$
\begin{equation}
\P_{\beta,a}^{IG}\left[X=k \right]\propto e^{-\frac{\beta}{2}(k-a)^2}\,,
\end{equation}
(recall that $\propto$ stands for proportional to). 
In this work, as in \cite{GS2}, we prefer to make reference to $(a,\beta)$ instead as $(\mu,\sigma^2)$.  Note that  $a$ is not the mean of $X$, nor $\beta^{-1}$ its variance. For $X$ an IV-Gaussian random variable of parameters $a$ and $\beta$, we denote
\begin{align}
&\mu^{IG}(a,\beta):= \E^{IG}_{\beta,a}\left[X\right],\\
&\var^{IG}(a,\beta):= \var_{\beta,a}^{IG}\left[X \right]\\
&T^{IG}(a,\beta):=\E^{IG}_{\beta,a}\left[|X-\mu^{IG}(a,\beta)|^3 \right].
\end{align}

\noindent
The following {\em error function} $\beta\mapsto M(\beta)$ will be used throughout in this text.
\begin{align}\label{e.M} 
M(\beta):= (2\pi)^2\beta\inf_{a\in[0,1/2]}\var^{IG}(a,(2\pi)^2\beta).
\end{align}

\noindent
We shall use the following estimates on $\var^{IG}(a,\beta)$, $T^{IG}(a,\beta)$ and $M(\beta)$ from Appendix B in \cite{GS2}.

\begin{proposition}[Appendix B in \cite{GS2}]\label{pr.Monotonicity_IG} $ $

\bi
\item[i)] For all $\beta>10$ and $a\in \R$
\begin{align}\label{e.variance_IG_bound}
\var^{IG}(a,\beta)\geq \frac{1}{16}e^{-\frac{\beta(1-2a)}{2}}\,.
\end{align}
\item[ii)] For any $\beta>0$, 
\begin{align}\label{e.K_beta}
K_\beta:= \sup_{\hat \beta> \beta} \sup_{a\in \R} \frac{T^{IG}(a,\hat \beta)}{\var^{IG}(a,\hat \beta)}\in(0,\infty)
\end{align}
\item[iii)]  For any $\beta\geq\tfrac13$, 
\begin{align}\label{}
M(\beta) \geq 2\beta \exp\left (-\frac{(2\pi)^2}{2} \beta\right )\,.
\end{align}
\ei
\end{proposition}
\noindent
As discussed in Appendix B of \cite{GS2}, we expect that for all $a\in \R$
\begin{align}\label{e.monotonicity_variance}
&\var^{IG}(0,\beta)\leq \var^{IG}(a,\beta)\leq \var^{IG}(0.5, \beta)\,,
\end{align}
which in turn would imply $M(\beta) \sim 2(2\pi)^2 \beta e^{-\frac{(2\pi)^2} 2 \beta}$.

\subsection{A reminder on discrete differential forms on $\Z^4$.}
\label{ss.reminder}

In this subsection, we give a short presentation of discrete differential calculus based on \cite{Roland,Sourav} as well as our companion paper \cite{GS2}. These notions are rather classical, yet we include these here as our way of handling boundary conditions differ slightly from other references. Also we will provide a detailed description of the Laplacian operator on $1$-forms in Section \ref{s.Laplacian_1} which will be of central importance throughout this text. For further useful references, see \cite{FSrestoration,gross1983convergence,driver1987convergence,GP,Roland,DarioWu,Sourav,malin,cao}. 

\subsubsection{Graphs and $k$-cells.}\label{sss.graphs} Many of the results in this paper are not concerned with a specific lattice. However, in order to keep notations light, we will only work with the following two types of graphs $\Lambda$.
\begin{itemize}
	\item The infinite volume case $\Lambda= \Z^n$, in most cases with $n=4$.
	\item The finite volume cubes $\Lambda = \Lambda_j =[-j,j]^n  \subseteq \Z^n$. 
\end{itemize} 
For each $0\leq k \leq n$,  \textbf{$k$-cells} of these graphs are obtained as the non-trivial intersection of $n-k+1$ unitary hyper-cubes\footnote{In contrast to the context of \cite{GS2}, we will not consider the complement of $[-j,j]^{n}$ as an $n$-cell.}. We consider $k$-cells as oriented objects (with either positive or negative orientation, see below) and we shall denote by $\overrightarrow C^k=\overrightarrow C^k(\Lambda)$ the set of $k$-cells of $\Lambda$ and by $C^k=C^k(\Lambda)$ as the non-oriented $k$-cells. 

Defining a suitable and consistent concept of orientations of $k$-cells (so that ultimately $\d^2=0$) is a rather delicate affair whose roots lie in the origins of differential exterior calculus. We will not make a self-contained presentation here. Instead we only briefly sketch below how it works and refer to \cite{gross1983convergence,driver1987convergence,Sourav} for more complete expositions (see also \cite{GP} for a more general way to define orientations). 
Let us fix $e_1,\ldots,e_n$ to be the canonical basis of $\Z^n$, which we view in this paragraph as oriented edges. If one considers a non-oriented $k$-cell which is based, say at some $x\in \Z^n$ and is spanned by $k$ vectors of the basis $v_1,\ldots,v_k=e_{i_1},\ldots,e_{i_k}$ with $i_1<\ldots < i_k$, then we have two possible oriented $k$-cells associated to it:
\bi
\item the positive cell $w$ corresponding in exterior diff. notations to $(v_1 \wedge v_2 \ldots \wedge v_k)_x$
\item its negative (or inverse) cell, which will be denote $w^{-1}$ and which corresponds to $-(v_1\wedge v_2 \ldots  \wedge v_k)_x$
\ei
To any such $k$-cell $w$, one define its \textbf{boundary} $\p w$ by 
\[
\p w = \sum_{\eps \in \{0,1\}} \sum_{j=1}^k (-1)^{\eps+j} (v_1 \wedge v_2 \wedge \ldots \wedge \hat v_j \wedge \ldots \wedge v_k)_{x+\eps v_j} 
\]
(where the singled out basis vector is to be omitted). The boundary $\p w$ may either refer to this formal sum (which will correspond below to a $k-1$-form) or to the collection of the $2k$ cells $(v_1 \wedge v_2 \wedge \ldots \wedge \hat v_j \wedge \ldots \wedge v_k)_{x+\eps v_j}$ equipped with their respective orientation $(-1)^{\eps+j}$. In fact, to simplify the notation we say that a $k-1$ cell $v$ belongs to a $k$-cell $w$ if $v$ is in $\partial w$.

Let us now describe certain types of $k$-cells
\begin{itemize}
	\item \textbf{0-cell} are the vertices of $\Lambda$. They are oriented positively or negatively. 
	\item \textbf{1-cell} are the oriented edges of $\Lambda$. They contain a positively oriented vertex and a negatively oriented one.
	\item \textbf{higher-dimensional cells.} See Figure \ref{f.kCell} for an illustration of the $3$-cell $(e_1\wedge e_2 \wedge e_3)_0$ oriented positively.
\end{itemize}
\begin{figure}[!htp]
\begin{center}
\includegraphics[width=\textwidth]{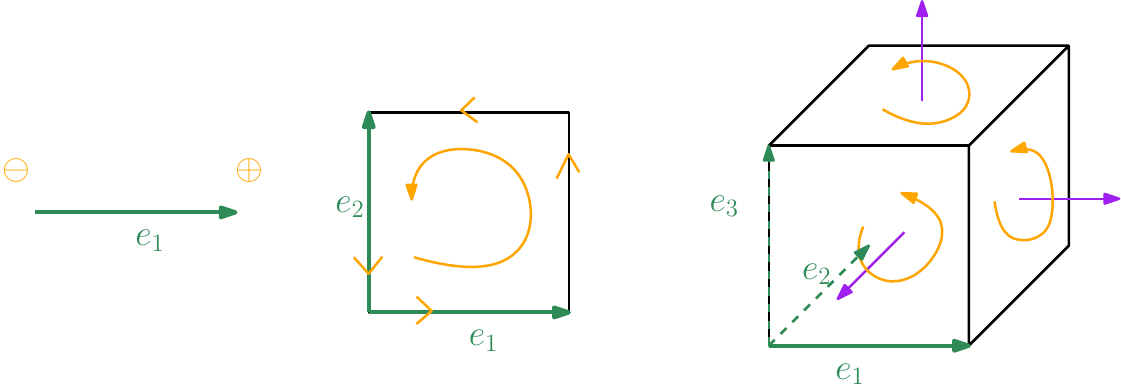}
\end{center}
\caption{}\label{f.kCell}
\end{figure}

For convenience (in particular when dealing with scalar-products), we identify the non-oriented $k$-cells in $C^k=C^k(\Lambda)$ with the positively oriented cells in $\overrightarrow C^k(\Lambda)$. (Note that once a basis $e_1,\ldots, e_n$ is fixed, this indeed singles out this way half of the $k$-cells).

Finally, when $\Lambda=\Lambda_j=[-j,j]^n\cap \Z^n$, we will say that a $k$-cell belongs to $\partial C^k$, the boundary of the graph $\Lambda$, if the whole $k$-cell is contained in the boundary of $[-j,j]^n$.

\subsubsection{Discrete differential calculus.}
We recall in this subsection the concepts of differential form and exterior derivative. 

 \begin{definition}[$k$-form] A function $f:\overrightarrow C^k\mapsto \R$ is a $k$-form if for all oriented $k$-cell $w$, we have that
 	\begin{align*}
 	f(w)=-f(w^{-1}).
 	\end{align*}
 	
 	We call $\Omega^k$ the set of $k$-forms, and $\Omega^k_\Z$ the set of integer-valued $k$-forms. Furthermore, we call $\mathring \Omega^k$ the set of $k$-form that take values $0$ in $\partial C^k$, and $\mathring \Omega^k_\Z$ the set of integer-valued $k$-forms taking $0$ value in $\partial C^k$. Note that $\Omega^n=\mathring \Omega^n$.
 \end{definition} 
The set of $k$-forms $\Omega^k$ is equipped with the following inner-product (which extends to $\mathring \Omega^k$).  For any $f_1, f_2 \in \Omega^k$,
\begin{align}\label{e.intern_product}
\langle f_1, f_2 \rangle := \frac{1}{2}\sum_{\overrightarrow w \in \overrightarrow{C}^k} f_1(\overrightarrow w) f_2(\overrightarrow w)= \sum_{ w \in C^k}f_1(w) f_2(w)\,, 
\end{align}
where recall that we identified above $C^k$ with the subset of $\overrightarrow{C}^k$ made of positively oriented cells.

We now define an operator $\d$, the \textbf{discrete exterior derivative}, that transforms a $k$-form into a $k+1$-form, in the following way:
 for $k\leq n-1$, $f\in \Omega^k$  and $\omega\in \overrightarrow{C}^{k+1}$
\begin{align*}
&\d f (w) = \sum_{v \in w} f(v).
\end{align*}
For $f\in \Omega^n$,
we set $\d f =0$. Note that $\d$ can be seen both as a linear operator from $\Omega^{k}$ to $\Omega^{k+1}$ 
as well as from $\mathring \Omega^{k}$ to $\mathring \Omega^{k+1}$. 

As $\d$ is a linear function, it can be thought of as a matrix. In this context, we define\footnote{The choice of the minus sign is because we want our Laplacian to be negative definite, as in analysis.} $\d^*: \Omega^k \mapsto \Omega^{k-1}$ as $-\d^t$. That is to say, for any $k$-form $f$ with $k\geq 1$ and any $w\in \overrightarrow C^{k-1}$
\begin{align*}
&\d^*f(w)=-\sum_{ v\ni w} f(v).
\end{align*}
We define $\d^* f=0$ for $f\in\Omega^0$. Furthermore, we also define $\mathring \d^*: \mathring \Omega^{k}\mapsto \mathring\Omega ^{k-1}$ as $-\d^t$, but this time the transpose is taken in the space of forms with $0$-boundary. This means that for any $\mathring f \in \mathring\Omega^{k}$ and $w\in \overrightarrow C^{k-1}$
\begin{align*}
\mathring \d^* \mathring f (w) :=\begin{cases}
-\sum_{v\ni w} \mathring f(v)&\text{ if } w\notin \partial \overrightarrow C^{k-1}\\
0 & \text{ if } w\in \partial \overrightarrow C^{k-1}
\end{cases}
\end{align*}
Note that for generic $\mathring f\in\mathring\Omega^k$ and $w\in \partial \overrightarrow{C}^{k-1}$ one has $\mathring \d^*\mathring f (w) \neq \d^* \mathring f(w)$ as there may be $v\in \overrightarrow{C}^{k}\backslash \partial \overrightarrow{C}^{k}$ such that $w\in v$.

The main usefulness of the operators $\d$ and $\d^*$ is given in the following classical proposition. (As explained below, it can be found for example in \cite{Sourav}).
\begin{proposition}\label{p.basic_calculus}The following statements are true for finite graphs $\Lambda\subseteq \Z^n$
	\begin{enumerate}
		\item $\d \d=0$.  In particular for all $1\leq k \leq n-1$, if $f\in \Omega^{k-1}$ (resp  $f\in \mathring\Omega^{k-1}$) and $g\in \Omega^{k+1}$ (resp. $\mathring g \in \mathring \Omega^{k+1}$) 
\begin{equation*}
\langle \d f, \d^*g\rangle =0 \text{ and } \langle \d \mathring f, \mathring \d^* \mathring g\rangle =0
\end{equation*}
		\item If $k\geq 1$ and $f\in \Omega^k$ (resp. $\mathring f\in \mathring \Omega^k$) is such that $\d f =0$ (resp. $\d \mathring f=0$), then there exists $g\in \Omega^{k-1}$ (resp. $\mathring g \in \mathring\Omega^{k-1}$) such that $\d g= f$  (resp. $\d \mathring g= \mathring f$).
	
		\item If $k\leq n-1$ and $f\in \Omega^k$ (resp. $\mathring f\in \mathring \Omega^k$) is such that $\d^* f =0$ (resp. $\mathring \d^* \mathring f(w)=0$), then there exists $g\in \Omega^{k+1}$ (resp. $\mathring g \in \mathring\Omega^{k+1}$) such that $\d^* g= f$  (resp. $\mathring \d^* \mathring g(w)= \mathring f(w)$). 
	\end{enumerate}	
\end{proposition}

\noindent
{\em Proof.}
	\begin{enumerate}
		\item  This result is Lemma 2.1 of \cite{Sourav}.
		\item This result is Lemma 2.2 of \cite{Sourav}.
		\item This result follows from the last item by using the discrete Hodge dual (see Section 2.6 of \cite{Sourav}). 
	\end{enumerate}
	


\begin{remark}\label{r.casek=0,n}
	Point (2) and (3) of the last proposition can also be studied for $k=0$ and $k=n$ respectively.
	\begin{itemize}
		\item If $f\in \Omega^0$ is such that $\d f=0$ then $f$ is a constant. Thus if $\mathring f \in \mathring \Omega^0$ is such that $\mathring \d \mathring f=0$, then $f$ has to be 0.
		\item  If $f \in \Omega^n$ is such that $\d^* f =0$ then $f=0$, 
		 however if $\mathring f\in \mathring\Omega^n=\Omega^n$ is such that $\mathring\d^*\mathring f=0$ we can only have that $\mathring f$ is constant.
	\end{itemize}
In the present work we do not use Proposition \ref{p.basic_calculus} in the cases where $k=0$ or $k=n$, and thus we will not extend this discussion. However, in some cases it may be useful to define a root vertex and a root $n$-cell which ones define to be $0$ as in \cite{GS2}.
\end{remark}

In this paper, we shall need the following improvement of the second point of the last proposition. We define the following equivalence class on $k$-forms taking values in the integers, i.e., on  $\Omega^k_\Z$
\begin{align*}
f_1 \mathcal R f_2 \text{ if } \d f_1 = \d f_2.
\end{align*}
We will denote by $[f]$ the equivalence class of $f$ under $\mathcal R$, and for $\mathring f \in \mathring \Omega^k_{\Z}$ we denote $[\mathring f]_{\ccirc}= [\mathring f] \cap \mathring \Omega^k_{\Z}$.

We also need to fix once and for all a deterministic function that satisfies the second point of Proposition \ref{p.basic_calculus}.
\begin{definition}\label{d.n_q}Take a finite graph $\Lambda\subseteq \Z^n$ and $k\geq0$. For any $q\in \Omega^{k+1}_\Z$ with $\d q=0$, we fix a deterministic function $n_q\in \Omega^{k}_\Z$ such that\footnote{The notation of $n_q$ is borrowed from  \cite{Roland} and \cite{GS2} to simplify the lecture of both papers.}  \[\d n_q =q,\]
	and such that if $\mathring q\in \mathring \Omega^{k}_\Z$ then $n_{\mathring q}\in \mathring\Omega^{k-1}_{\Z}$.
\end{definition}
The equivalence class and the definition above were introduced in order to state the following bijection.
\begin{proposition}\label{p.bijection}
	Take $k\geq 1$. There exists a bijection between the integer-valued $k$-forms $f\in \Omega^k_\Z$, and the Cartesian product of $k+1$-forms $q\in \Omega^{k+1}_{\Z}$ with $\d q=0$ and the equivalence classes of $k-1$-forms $[\psi]$. This bijection is given by
	\begin{align*}
	f= \d \psi' + n_{q}, \ \ \text{ for a } \psi'\in [\psi]. 
	\end{align*}
	Furthermore, this bijection is extended to the case where $\mathring f\in \Omega^k_{\Z}$, $\mathring q\in \mathring\Omega^{k+1}_{\Z}$  and $[\mathring \psi]_{\ccirc}$ is an equivalence class of $\mathring \Omega^{k-1}_{\Z}$.
\end{proposition}
\begin{proof}
	We start working with the case of free-boundary condition. We see that the function $F:([\psi],q)\mapsto f$ is a bijection. We first see that it is an injection assume \begin{align*}
F([\psi],q)=f = \tilde f = F([\tilde \psi],\tilde q).
	\end{align*}
	We have that $\d f = \d \tilde f$ and thus, $q = \tilde q$. This implies that $n_q= n_{\tilde q}$ and thus $\d \psi' = \d \tilde \psi'$, from where we conclude that $[\psi]=[\tilde \psi']$.
	
	To prove that $F$ is surjective, we just need to take $f\in \Omega^{k}_\Z$ and define $q=\d f$. Noting that
	$\d(f-n_q)=0$ we have  that, thanks to Proposition \ref{p.basic_calculus}, there exists $\psi\in \Omega^k_\Z$ such that
	\begin{align*}
	\d \psi = f-n_q,
	\end{align*}
	from where we conclude.
	
	The same proof works for the case of $0$-boundary conditions.
\end{proof}

Finally, let us define the Laplacian operator on $\Omega^k$ and $\mathring \Omega^k$ as
\begin{align*}
\Delta= \d \d^*+ \d^*\d\ \ \  \text{ and } \ \ \ \mathring \Delta = \d \mathring \d^* + \mathring \d^* \d.
\end{align*}
Notice that the Laplacian commutes both with $\d$ and $\d^*$, and $\d$ and $\mathring \d^*$ respectively

 We state the following well-known result concerning the Laplacian operator.
\begin{proposition} Let $\Lambda\subseteq \Z^n$ be a finite graph. If $0<k<n$, then the Laplacian operator is (strictly) negative definite on $\Omega^k$. In particular, the Laplacian is an invertible operator. Furthermore, the same is true for $\mathring \Delta$ on $\mathring \Omega^k$.
\end{proposition}

\begin{proof}
	We first prove that $\Delta$ is negative definite and then we show that it is invertible. To do that, let us first note that
	\begin{align}
	\label{e.neg_def}\langle (\d \d^* + \d^* \d) f, f \rangle&= -\langle \d^* f, \d^*f \rangle-\langle \d f, \d f \rangle \leq 0.
	\end{align}
	Thus $\Delta$ is negative semi-definite. Furthermore, assume that $\Delta f=0$, to finish the proof of the proposition it sufices to show that $f=0$. To do that, note that if $\Delta f =0$, we have that thanks to \eqref{e.neg_def},  both $\d f=0$ and $\d^*f=0$. Thus, there exists $g^-$ and $g^+$ such that $\d g^- = \d^* g^+ = f$. This implies that
	\begin{align*}
	\langle f , f\rangle = \langle \d g^-, \d^* g^+\rangle =0,
	\end{align*}
	which allows us to conclude.
	
 The exact same proof works for $\mathring \Delta$ thanks to Proposition \ref{p.basic_calculus}. 
\end{proof}

\begin{remark}\label{r.laplacian0,n}
	In fact, the proof of this proposition can be directly extended for the case $k=0$ when one works with $\mathring \Omega^{0}$ and for the case $k=n$ for $\Omega^n$. However, the proof does not work directly for $k=0$ in $\Omega^0$ and for $k=n$ in $\mathring \Omega^{n}$. The reason is explained in Remark \ref{r.casek=0,n}, and again it can be solved by fixing a root edge and a root $n$-cell as in \cite{GS2}.
\end{remark}

\subsection{Relevant statistical physics models.}

\subsubsection{Villain $U(1)$-lattice gauge theory.}
Thanks to the previous Section, we may now rewrite our initial definition of the Villain model (in Definition \ref{d.Villain}) using the language of discrete differential calculus.  Following the new input from \cite{GS2}, we will at once make one step further by extending the classical definition of the  Villain $U(1)$-lattice gauge theory to a probability measure on   couplings $(\theta,m)$ where $m$ is a random 2-form and $\theta$ is a one-form whose marginal law corresponds to the Villain interaction of Definition \ref{d.Villain}. Here is the definition of this joint coupling (where we use the same notation $\P_\beta^{Vil}$ with a slight abuse of notations).

\begin{definition}[Villain $U(1)$-lattice gauge coupling]\label{d.Villain_final} Let $\Lambda\subseteq \Z^4$ be a finite graph. We say that a pair $(\theta,m)$ is a Villain $U(1)$-lattice gauge coupling with free boundary condition if $\theta \in \Omega^1$ is a $1$-form taking values in $[-\pi,\pi)$, $m$ is an integer-valued $2$-form in $\Omega_\Z^2$ and
\begin{align}\label{e.uvil}
\P_{\beta}^{\Vil}((d\theta,m))\propto e^{-\frac{\beta}{2}\langle \d \theta + (2\pi)m, \d \theta + (2\pi)m\rangle} d\theta.
\end{align}

A pair $(\theta,m)$ is a Villain $U(1)$-lattice gauge coupling with zero boundary condition if $\theta \in \mathring \Omega^1$ takes values in $[-\pi,\pi)$, $m\in \mathring \Omega^2_\Z$ and it also satisfies \eqref{e.uvil}.

\end{definition} 

The following straightforward proposition highlights the key fact that conditioned on $\theta$, $m$ is an inhomogeneous discrete white noise on the 2-cells. 
\begin{proposition}\label{pr.Villain}Let $(\theta, m)$ be a Villain $U(1)$-lattice gauge theory (with any boundary conditions). We have that conditionally on $\theta$ the collection of random variables $(m(f))_{f\in C^2}$ are independent. Furthermore, the law of $m(f)$ conditioned on $\theta$ is that of an IV-Gaussian random variable at inverse-temperature $(2\pi)^2\beta$ and centred at $-(2\pi)^{-1}\d\theta(f)$ (except when $f\in \p C^2$ and the model has zero boundary conditions. In this case, $m(f)$ is $0$).
\end{proposition}

\begin{proof}
		 The result comes directly from writing down the conditional law
		\begin{align*}
		\P_{\beta}^{Vil}(m \mid \d \theta )\propto \prod_{f \in C^2} \exp\left(-\frac{\beta}{2}(\d \theta(f) + 2\pi m(f))^2 \right) \delta_\Z(dm). 
		\end{align*}
\end{proof}

We now discuss the infinite volume limit of $(\theta,m) \sim \P_{\beta,\Lambda_j}^{\Vil}$.
As this is not necessary for our result, we only give a sketch of proof below.
\begin{proposition}\label{pr.IVL}
Let $\beta>0$ and $(\theta_j,m_j)$ be a Villain $U(1)$-lattice gauge theory in $\Lambda_j= [-j,j]^n \cap \Z^n$ with free boundary conditions  at inverse temperature $\beta$. Then as $j\to \infty$, $(\theta_j,m_j)$ converges in law to an infinite-volume $U(1)$-lattice gauge theory on $\Z^n$.
\end{proposition}

\noindent
{\em Sketch of proof.} 
As the conditional law of $m_j$ given $\theta_j$ is local, it is enough to obtain the infinite volume limit only for the $1$-form $\theta_j$. This is stated both for the Wilson and Villain interaction in \cite{FSrestoration} as a standard consequence of Ginibre's inequalities (\cite{Ginibre}). See also on this topic the works \cite{king1986u2,driver1987convergence} as well as the following two useful references: \cite{messager1978}  in the case of the standard $XY$ model on $\Z^2$ and the recent \cite{malin} which provides a clear and self-contained proof when the gauge group is a discrete Abelian group.  \qed

{\em N.B. The convergence to a (possibly different) infinite volume limit in the case of Dirichlet boundary conditions should also follow from the same Ginibre's inequality, but some further care may be needed as the most natural condition along the boundary to use Ginibre's inequality would be to require $\d\theta (f)=0$ on boundary plaquettes $f$ rather than $\theta (e) =0$ on boundary edges. In any case, by compactness, note that one can always state our results for any infinite volume subsequential limits.}


\subsubsection{The (standard) Gaussian free field on $1$-forms.}\label{ss.GFF}
The definition below is the discrete version of the so-called {\em free electromagnetic wave} (\cite{gross1983convergence,driver1987convergence}). 
\begin{definition}\label{d.SGFF}
The $\beta$-GFF on $1$-forms is the real-valued centred Gaussian process $\phi \sim \P_{\beta,\Lambda}^{\GFF}$ on $\Omega^1(\Lambda)$ for free-boundary condition and in  $\mathring \Omega^1(\Lambda)$ for $0$-boundary condition whose covariance matrix is given by 
\begin{align*}
\begin{cases}
\frac 1 \beta \, (-\Delta)^{-1} = \frac{1}{\beta}(- (\d\d^* + \d^*\d))^{-1}\, & \text{ for free-boundary condition},\\
\frac 1 \beta \, (-\mathring \Delta)^{-1} = \frac{1}{\beta}(- (\d \mathring \d^* + \mathring \d^*\d))^{-1}\, & \text{ for $0$-boundary condition},
\end{cases}
\end{align*}
 
In other words, we say that $\phi$ is a GFF on the $1$-forms if it belongs to either $\Omega^1$ (for free boundary condition) and $\mathring \Omega^1$ (for $0$-boundary condition) and whose probability distribution is given by
\begin{align}
\P_{\beta}^{\GFF}\left[d\phi \right]\propto\begin{cases}
 \exp\left(-\frac{\beta}{2}\langle \phi ,(-\Delta) \phi \right)d\phi & \text{ for free boundary condition,}\\
 \exp\left(-\frac{\beta}{2}\langle \phi ,(-\mathring \Delta) \phi \rangle\right) d\phi & \text{ for zero boundary condition}.
 \end{cases}
\end{align}
\end{definition}
See also \cite{FSrestoration, gross1983convergence,driver1987convergence,Roland}.

It turns out that the spin-wave which will naturally arise for $U(1)$ lattice gauge theory on $\Lambda \subset \Z^4$ is not quite this Gaussian process, but rather its push-forward image under $\d$ as we shall explain in the following section. This will be particularly important in Section \ref{s.decoupling}, see Remark \ref{r.ImpSpW}. 

\subsubsection{The  gradient spin-wave on 2-forms.}
In the case of the classical Villain model in $2D$, one may view the spin-wave either as a Gaussian field on the vertices (the GFF) or as a Gaussian field on 1-forms (the gradient of the GFF). In that case, both point of views happen to be equivalent. On the other hand, in the present setting both point of views are no longer equivalent as there is a lot of information lost when one takes the discrete differential $\d$ of a GFF. It is in fact this second choice which will lead to a decoupling of the spin-wave.

Recall that the Villain-$U(1)$ lattice gauge theory introduced in~\eqref{e.Villain_mar} is a periodized Gaussian on the set $[-\pi,\pi)^{\overrightarrow C^1(\Lambda)}$.   By getting rid of the periodization in~\eqref{e.Villain_mar} and taking its discrete differential $\d$, we obtain the following Gaussian process on 2-forms (which we call from now on the {\em gradient spin-wave}).
\begin{definition}[The gradient spin-wave on $2$-forms]\label{d.SW}
	We say that $\varrho$ is a gradient spin-wave on $2$-forms at inverse temperature $\beta$ if $\varrho \stackrel{law}{=} \d \phi$, where $\phi$ is a GFF on $1$-forms at inverse temperature $\beta$. The boundary condition of the gradient spin-wave is inherited from the boundary condition of the GFF.
\end{definition}

The main reason to study the gradient spin-wave instead of the GFF itself is that it will naturally appear when decoupling the angles of the Villain-$U(1)$ lattice gauge theory into a suitable spin-wave and a Coulomb gas below in Proposition \ref{p.decouplingF}. 
As opposed to the classical $2d$ Villain model (on $0$-forms), this does not imply in our case that there exists a decoupling at the level of $1$-forms involving the GFF itself. 

\section{Laplacian of $1$-forms}\label{s.Laplacian_1}

In this section, we study the Laplacian on $1$-forms. Our analysis could easily be extended to $k$-forms with $1\leq k\leq n-1$ but we will not need it. We start by  stating some basic results concerning the Laplacian on $0$-forms.
\subsection{Laplacian on $0$-forms: zero boundary condition.}\label{ss.Laplacian_0}

Recall the Laplacian on $0$-forms is not invertible on $\Omega^0$ as the constant functions are in the kernel of $\d$, however  $\mathring \Delta$ is invertible as a function on $\mathring \Omega^0$. 

The inverse of the Laplacian on $\mathring \Omega^0$ is called the Green's function, $G_{\Lambda_j}:=(-\mathring \Delta)^{-1}$. Furthermore as $\Lambda_j\nearrow \Z^n$ this solution converges to the classical Green's function $G$ on $\Z^n$.  We will need some control on the asymptotical behaviour of this Green's function. The following proposition will be sufficient for our needs and follows from classical estimates on Green functions in $\Z^n$ (see for example \cite{LawlerLimic}).
\begin{proposition}\label{p.Green_vertices}
	Let $G=(-\mathring \Delta)^{-1}$ be the Green's function of the Laplacian on the $0$-forms of $\Z^n$, for $n>2$. For any $x,y$ 
	\begin{align*}
	0<G(x,y) = C_n \|x-y\|^{-(n-2)} + O(\|x-y\|^{-n}).
	\end{align*}
\end{proposition}

Furthermore, the following lemma compares the Green's function in $\Z^n$ with the Green's function in a smaller graph $\Lambda$.
\begin{lemma}\label{l.approximation}
For all points $x,y$ in a finite graph $\Lambda \subset \Z^n$ that are at $\ell^2$ distance more than $M$ from $\partial \Lambda$ we have 
	\begin{align*}
	G (x,y)-O\left (\frac{1}{M^{n-2}}\right ) \leq G_{\Lambda}(x,y) \leq G(x,y).
	\end{align*}
\end{lemma}
{\em Proof.}
	The upper bound follows readily from the monotonicity of the Green's function. The lower bound is obtained from the fact that $G_\Lambda(x,\cdot)$ is the unique harmonic function in $\Lambda\backslash (\partial \Lambda\cup \{x\})$ that is $0$ in $\partial \Lambda$ and so that $\Delta (G_\Lambda(x,\cdot)) (x)=-1$, and thus 
	\begin{align*}
	G_{\Lambda}(x,y) \geq G(x,y)- \sup_{z} G(x,z).
	\end{align*}


\subsection{Laplacian on 1-forms.} Our goal in this subsection is to understand the behaviour of the Laplacian on the $1$-forms by relating it to the Laplacian on $0$-forms. To do this, we  introduce new graphs $(\G_i)_{i=1}^n$.

We start by considering the set of oriented edges of $\Z^n$, $\overrightarrow C^1(\Z^n)$. Take $\{e_i\}_{1\leq i \leq n}$ to be the canonical base of $\Z^n$ and let $\overrightarrow{e_i}$ be the edge going from $0$ to $e_i$. Now, we define the subset $E_i\subseteq \overrightarrow C^1(\Z^n)$ of oriented edges that have the same orientation as $\overrightarrow{e_i}$, i.e, $e\in E_i$ if $e=k+ \overrightarrow{e_i}$, with $k\in \Z^n$. Furthermore, we define a graph $\G_{i}$ whose vertices are given by $E_i$  and where there is an edge between $e,e' \in E_i$ if there exists $1\leq j \leq n$
such that
\begin{align}
e=e'\pm e_j.
\end{align}
This concludes the definition of the infinite graphs $(\G_i)_{i=1}^n$. In order to adapt the definition to a finite setting we discuss separately  both types of boundary conditions.

\subsubsection{Free-boundary conditions.} For a finite graph $\Lambda\subseteq \Z^n$, we define the graph $\mathcal G_i(\Lambda)$ as the subgraph of $\G_i$ generated by all the vertices (i.e. edges of the initial graph) of $\G_i$ that intersect $\Lambda$.
Furthermore, we define $\partial \G_i$ as all the vertices of $\G_i(\Lambda)$ that do not belong to $\overrightarrow C^1(\Lambda)$. (See Figure \ref{f.E_i}).

We shall also consider $\Delta_{\mathcal G_i(\Lambda)}$  the Laplacian on the vertices of $\mathcal G_i(\Lambda)$ with 0 boundary conditions on $\partial \mathcal G_i(\Lambda)$. I.e., for any function $f:\mathcal G_i(\Lambda)\mapsto \R$ taking values $0$ in $\partial \mathcal G_i(\Lambda)$ we define
\begin{equation}\label{e.Laplacian_Gi}
\Delta_{\mathcal G_i(\Lambda)} f (e) := \sum_{e'\sim e}f(e')-f(e), \ \ \ \ \forall e\in \mathcal G_i(\Lambda)\backslash \partial \mathcal G_i(\Lambda).
\end{equation}

\begin{figure}[h!]
	
	\includegraphics[width=0.4\textwidth]{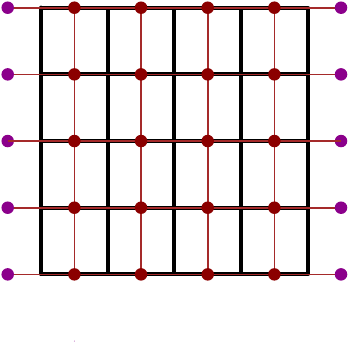}
	\caption{Representation of a finite graph $\Lambda\subseteq \Z^2$. The graph $\mathcal G_1(\Lambda)$ is represented by the dots, the interior of the graph is given by the red dots and the boundary is represented by the purple ones. }
	\label{f.E_i}
\end{figure}

\subsubsection{$0$-boundary conditions.}
Let us now study the finite graph $\Lambda \subseteq \Z^n$ with $0$-boundary condition. We define the graph $\mathring \G_i(\Lambda)$ as the subgraph of $\G_i$ generated by all vertices $e\in \G_i$ that belong to $\overrightarrow C^{1}(\Lambda)$. We define $\partial \mathring \G_i(\Lambda)$ as the set of edges of $\mathring \G_i(\Lambda)$ which belong to $\partial \overrightarrow C^{1}(\Lambda)$. (See Figure \ref{f.E_i_0}).
\begin{figure}[h!]
	
	\includegraphics[width=0.4\textwidth]{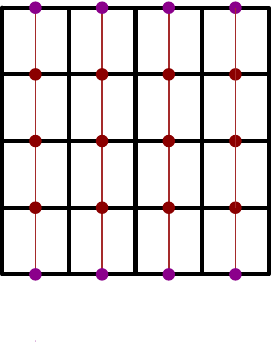}
	\caption{Representation of a finite graph $\Lambda\subseteq \Z^2$. The graph $\mathring\G_1(\Lambda)$ is represented by the dots, the interior of the graph is given by the red dots and the boundary is represented by the purple ones. }
	\label{f.E_i_0}
\end{figure}

We consider $\Delta_{\mathring \G_i(\Lambda)}$ the Laplacian on the vertices of $ \mathring \G_i(\Lambda)$ with 0-boundary conditions on  $\partial \mathring \G_i(\Lambda)$. I.e., for any function $f: \mathring \G_i(\Lambda)\mapsto \R$ taking values $0$ in $\partial \mathring \G_i(\Lambda)$ we define
\begin{equation}\label{e.Laplacian_Gi0}
\Delta_{\mathring \G_i(\Lambda)} f (e) := \sum_{e'\sim e}f(e')-f(e), \ \ \ \ \forall e\in \mathring \G_i(\Lambda)\backslash \partial \mathring \G_i(\Lambda).
\end{equation}

\subsection{Relationship between the Laplacian on $1$-forms and the Laplacian on $0$-forms.}Let us be more explicit with equations \eqref{e.Laplacian_Gi} and \eqref{e.Laplacian_Gi0}. Take an edge $e\in [-j+1,j-1]^n\cap \Z^n$, and note that \eqref{e.Laplacian_Gi} becomes in this case
	\begin{equation}\label{e.Laplacian_int}
	\Delta_{\mathcal G_i(\Lambda)} f (e) = \sum_{\sigma\in \{-1,1\}} \sum_{l=1}^n f(e+\sigma e_l )- f(e).
	\end{equation}
	
	We have to be more careful in the case where the edge $e$ intersects $\partial [-j,j]^n$. Assume first that $e$ intersects $\partial [-j,j]^n$ but it is not contained in it. In this case, we have that either $e+e_i$ or $e-e_i$ does not belong to $\Lambda_j$. This case also comes back to \eqref{e.Laplacian_int} by recalling that $f(e+e_i)$ or $f(e-e_i)$ has to be $0$.
	
	Finally, we study the case where $e\subseteq \partial [-j,j]^n$, this case is only relevant for the free-boundary condition. Here we have that for some $l\neq i$, either $e+e_l$ or $e-e_l$ do not belong to the boundary. We then have  that
	\begin{align*}
		\Delta_{\mathcal G_i(\Lambda)} f (e) = \sum_{\sigma\in \{-1,1\}} \sum_{l=1}^n (f(e+\sigma e_l )- f(e)) \1_{e+ \sigma e_l  \in \G_i(\Lambda)}.
	\end{align*}

The same description applies to the case with $0$-boundary condition. 
The graphs we defined above are important thanks to the following result. (Recall the definition of $\mathring \Omega^1(\Lambda)$ from Subsection \ref{sss.graphs}).
\begin{proposition}
	Let $\Lambda\subseteq \R^n$ be a finite graph, $ f\in \Omega^1(\Lambda)$ and $\mathring f\in \mathring \Omega^1(\Lambda)$. We define the functions $f_i:\G_i\mapsto \R$ and $\mathring f_i:\mathring \G_i\mapsto \R$  as follows
	\begin{align*}
	&f_i(e)= f(e) \1_{e \in \overrightarrow C^{1}(\Lambda )},\\
	&\mathring f_i (e)= \mathring f(e) \1_{e \in \overrightarrow C^{1}(\Lambda )}.
	\end{align*}
	We then have that 
	\begin{align*}
	\Delta f(e) = \Delta_{\G_i(\Lambda) } f_i(e),& \text{\,\,   for any $e \in \G_i \backslash \partial \G_i$,}\\
	\mathring \Delta \mathring f (e) = \Delta_{\mathring\G_i(\Lambda) } \mathring f_i(e), &\text{\,\,   for any $e \in \mathring \G_i \backslash \partial \mathring \G_i$.}
	\end{align*}
\end{proposition}
\begin{proof}
	In this proof, we work with both cases simultaneously. Furthermore, without loss of generality we take $i=n$ and we denote $e=\overrightarrow{v_1v_2}$.

	Now, we study three different cases as in the beginning of this subsection according to where the edge $e$ is. 
	\begin{enumerate}[(1)]
\item  The first case is when $e$ is an edge of $\Lambda_{j-1}$. In this case, we have using $\Delta=\d\d^*+\d^*\d$ that 
	\begin{align}
	\label{e.Laplacian_1_form}\mathring \Delta f (e)=\Delta f (e) &= -\sum_{F\ni e}\sum_{e'\in F} f(e') - \sum_{e'\ni v_2}f(e') + \sum_{e' \ni v_1} f(e')\,,
	\end{align}
where, with a slight abuse of notations, $\sum_{F\ni e}$ stands for the sum over oriented $2$-cells $F$ which include the oriented edge $e$ in their boundary $\p F$ and similarly $\sum_{e'\ni v_1}$ stands for the sum over oriented edges $e'$ whose positive vertex is $v_2$. We now claim 
that \eqref{e.Laplacian_1_form} is equal to
	\begin{align}
	-2n f(e) +\sum_{\sigma\in \{-1,1\}} \sum_{l=1}^{n}f(e+\sigma e_l)  
	\nonumber&=\sum_{\sigma\in \{-1,1\}} \sum_{l=1}^{n}(f(e+\sigma e_l)- f(e)).
	\end{align}
We leave it to the reader to check this identity. Let us just explain in a few words the term $-2n f(e)$. One contribution comes from the $2(n-1)$ oriented faces $F$ which are such that $e\in \p F$. Each of these contribute one $-f(e)$. The remaining $-2f(e)$ comes from the edge $e'=e$ which appear twice in  $- \sum_{e'\ni v_2}f(e') + \sum_{e' \ni v_1} f(e')$. 
	
\item  The second case is when $e$ intersects $\partial [-j,j]^n$ but it is not contained in it. WLOG let us assume that $e+e_n$ does not belong to $\Lambda_j$,

	\textit{Free-boundary.} In this case, we still have 
	\begin{align}\label{e.Laplacian_1_form_b1}
	\Delta f (e) &= -\sum_{F\ni e}\sum_{e'\in F} f(e')- \sum_{e'\ni v_2}f(e') + \sum_{e' \ni v_1} f(e').
	\end{align}
	Let us note that we have $2(n-1)$ faces that contain $e$ but there is one less horizontal edge than in the previous case. Thus, we have that \eqref{e.Laplacian_1_form_b1} is equal to
	\begin{align}
	-2n f(e)+ f(e-e_n) +\sum_{\sigma\in \{-1,1\}} \sum_{l=1}^{n-1}f(e+\sigma e_l)  
	\nonumber&=\sum_{\sigma\in \{-1,1\}} \sum_{l=1}^{n}(f(e+\sigma e_l)- f(e)),
	\end{align}
	where we take $f(e+e_n)$ to be equal to $0$.
	
	\textit{Zero-boundary.} In this case $v_2 \in \partial [-j,j]^n$, and thus using that $\mathring \d^*\mathring f (v_2)=0$, we see that
	\begin{align}\label{e.Laplacian_1_form_b1z}
	\mathring \Delta \mathring f (e) &= -\sum_{F\ni e}\sum_{e'\in f} \mathring f(e') + \sum_{e' \ni v_1} \mathring f(e').
	\end{align}
	Using that the difference between \eqref{e.Laplacian_1_form_b1z} and \eqref{e.Laplacian_1_form_b1} has one additional term of $f(e)$, we have that \eqref{e.Laplacian_1_form_b1} is equal to
	\begin{align}
	-(2n-1) f(e)+ f(e-\sigma_n) +\sum_{\sigma\in \{-1,1\}} \sum_{l=1}^{n-1}f(e+\sigma e_l)  
	\nonumber
	\end{align}
	which is what we wanted. 
\item Finally, the last case is when $e$ is contained in $\partial [-j,j]^n$. This case is only relevant for free boundary conditions.  In this case, we know that \eqref{e.Laplacian_1_form_b1} also holds and there exists a subset of the basis, say $\{e_l\}_{l=1}^{l'}$, such that $e +e_l\nsubseteq \Lambda_j$. In this case, we note that there are only $2(n-1)-l'$ faces that contain $e$. This is exactly the number of neighbouring edges that $e$ has in $\G_n$. Thus, \eqref{e.Laplacian_1_form_b1} is equal to
	\begin{align*}
	 &-2(n-l') f(e)+\sum_{\sigma\in \{-1,1\}} \sum_{l=1}^{n}f(e+\sigma e_l)\1_{e+\sigma e_l \subseteq [-j,j]^n}\\
	 & \hspace{0.4\textwidth}=\sum_{\sigma\in \{-1,1\}} \sum_{l=1}^{n}(f(e+\sigma e_l)- f(e))\1_{e+\sigma e_l \in \G_n},
	\end{align*}
	where we again use the fact that $f(e+\sigma e_n)=0$ if $e+\sigma e_n$ is not an edge of $\Lambda_j$.
\end{enumerate}
\end{proof}

The above proposition allows us to obtain the following result regarding the Green's function on $1$-forms.
\begin{corollary}\label{c.Green_1}
	Let $\Lambda$ be a finite graph, as introduced in Section \ref{sss.graphs}. We have the following properties for any $e, e'$ two edges of $\Lambda$
	\begin{enumerate}[i)]
	\item If $e$ and $e'$ are both in the same direction as $\overrightarrow{e_i}$, i.e. $e,e'\in E_i$ we have that
	\begin{align*}
	&(-\Delta)^{-1}(e,e') = G_{\mathcal G_i(\Lambda)} (e,e')\\
	&(-\mathring \Delta)^{-1}(e,e') = G_{\mathring\G_i(\Lambda)} (e,e')
	\end{align*}

	\item Let $e,e'$ be two edges with different directions, i.e., $e \in E_i$ and $e' \in E_j$ with $i\neq j$, we have that
	\begin{align*}\label{}
	(-\mathring\Delta)^{-1}(e, e') = (-\Delta)^{-1}(e, e') = 0.
	\end{align*} 
\end{enumerate}
\end{corollary}

By taking the limit as $\Lambda_n \nearrow \Z^n$, this allows us to define the inverse of the Laplacian in the edges of $\Z^n$ as follows.

\begin{definition}\label{d.Green_1}
	Let $\Lambda= \Z^n$ and note that $\mathcal G_i$ is isomorphic to $\Z^n$. We define the Green's function on the $1$-forms as follows, take $e,e' \in \overrightarrow{C}^1$
	\begin{align*}
	(-\Delta)^{-1}(e,e') = \begin{cases}
	G_{\mathcal G_i } (e,e') & \text{ if } e,e'\in \mathcal G_i\\
	0 &\text{ if } e\in \mathcal G_i, e' \in \mathcal G_j, i\neq j.
	\end{cases} 
	\end{align*}
	Here $G_{\mathcal G_i }$ is the Green's function of the Laplacian in the graph $\mathcal G_i\approx \Z^n$.  
\end{definition}

We finish this section with a discussion on the convergence of the Green's function of $\G_i(\Lambda_j)$ to that of $\G_i$.
\begin{lemma}\label{l.approximation2}
There exists a constant $K$ such that for any $j\geq 1$ and any edges $e,e' \in \G_i(\Lambda_j)$ whose distance to $\partial [j,j]^n$ is greater than or equal to $M$, we have 
	\begin{align}
	&\label{e.app_Green_free}|G_{\G_i(\Lambda_j)}(e,e')-G_{\G_i}(e,e')| \leq \frac{K}{M^{n-2}}\\
	&\label{e.app_Green_0}|G_{\mathring \G_i(\Lambda_j)}(e,e')-G_{\G_i}(e,e')| \leq \frac{K}{M^{n-2}}.
	\end{align}
\end{lemma}
\begin{proof}
	We will prove \eqref{e.app_Green_free}, the other equation \eqref{e.app_Green_0} is similar. We start by recalling that the Green's function can be obtained as $(2n)^{-1}$ times the mean number of visits to $u_2$ by a random walk started from $u_1$ that is killed in $\partial \G_n (\Lambda_j)$ and is reflected in all other points of the boundary \footnote{For this identification, it is easier to work in a graph where all vertices have same degree equal to $2d$. To do that, one can add self-edges to each vertex with degree strictly smaller than $2d$ so that its degree is equal to $2d$.}. 
	This discussion, together with Lemma \ref{l.approximation}, implies that
	\begin{align*}
	G_{\mathcal G_i(\Lambda_j)}(e,e')\geq G_{\mathcal G_i ^0(\Lambda_j)}(e,e')\geq G_{\mathcal G_i}(e,e') \dgreen{-} O(M^{-(n-2)}),
	\end{align*}
	where $\mathcal G_i ^0(\Lambda_j)$ is the graph that puts $0$ boundary condition in all points where $\mathcal G_i (\Lambda_j)$ puts free boundary condition.
	
	For the upper bound, we use a reflection trick around all points where the random walk is reflected to obtain a new (infinite) graph $\widetilde{\G_n}(\Lambda_j)$. We denote $\tilde e$ and $\tilde e'$ the points of $\widetilde{\G_n}(\Lambda_j)$ that are identified to $e$ and $e'$ respectively. In this way
	\begin{align*}
	G_{\G_n(\Lambda_j)}(e,e')= \sum_{\tilde e, \tilde e'} 	G_{\widetilde \G_n(\Lambda_j)}(\tilde e, \tilde e').
	\end{align*}

	Note that for each (reflected) copy $\widetilde{\G}_n^{\ell}(\Lambda_j)\subseteq \widetilde{\G}_n(\Lambda_j)$ of $\G_n(\Lambda_j)$ there is one $\tilde e_\ell$ and $\tilde e_\ell'$ that is identified with $e$ and $e '$. To finish the proof we just need to show that there exists a constant $K$ such that
	\begin{align*}
G_{\widetilde \G_n(\Lambda_j)}(\tilde e_\ell, \tilde e'_\ell) \leq O(M^{-(n-2)}) e^{-K d_{Graph}(\widetilde{\G}_n^{\ell}(\Lambda_j), \G_n(\Lambda_j))} G_{\G_i}(e,e').
	\end{align*}
	Here, $d_{Graph}(\widetilde{\G}_n^{\ell}(\Lambda_j), \G_n(\Lambda_j))$ is the minimum amount of copies of $\G_n(\Lambda_j)$ that one has to cross in $\widetilde G_n(\Lambda_j)$ to go from $\G_n(\Lambda_j)$ to $\widetilde{\G}_n^{\ell}(\Lambda_j)$. This result follows from the following facts:
	\begin{enumerate}
		\item The probability that a random walk in $\Z^n$ started from $\partial [-j,j]^n$, the border of $\G_n(\Lambda_j)$, hits $e$ or $e'$ during its life time is less than or equal to $M^{-(n-2)}$.
		\item Take a random walk in $\Z^n$ that starts at the border of a $\G_n(\Lambda_j)$-like box and stops the first time it hits the border of another $\G_n(\Lambda_j)$-like box  that is at $d_{Graph}$ distance equal to $2$. One has that with uniformly positive probability (bigger than say $1/2n=1/8$) this random walk hits the 0-boundary of $\widetilde{\G}_n^{\ell}(\Lambda_j)$ during its life time. 
	\end{enumerate} 
\end{proof}

\section{Orthogonal decomposition of $k$-forms}\label{s.OD}
We will need the orthogonal decomposition of $k$-forms induced by the linear operators $\d$ and $\d^*$. We will only need to deal with finite cubes $\Lambda=\Lambda_j \subset \Z^n$. See Remark \ref{r.Odecomp} for a discussion in the infinite volume case. 

We start by defining the following vector subspaces.

%

\begin{definition}
	Let $\Omega^{k}$ be the space of real-valued $k$-forms in a finite cube $\Lambda$. We define
\begin{align}
	\label{e.k-1_k}&\Omega^{k-1\rightarrow k} = \d(\Omega^{k-1})  \\
	\label{e.k+1_k}&\Omega^{k+1 \rightarrow k} = \d^*(\Omega^{k+1}) .
	\end{align}
In the case of zero-boundary conditions, we also define 	
	\begin{align}
	\label{e.k-1_k_c}&\mathring \Omega^{k-1\rightarrow k} = \d(\mathring \Omega^{k-1}) \text{ and }\mathring \Omega^{k+1 \rightarrow k} = \mathring \d^*(\mathring \Omega^{k+1}).
	\end{align}
	
	Using this, we define $\proj{k-1}{k}, \proj{k+1}{k}: \Omega^k \mapsto \Omega ^k$ as the orthogonal projection of $\Omega^k$ into $\Omega^{k-1\rightarrow k}$ and $\Omega^{k+1\rightarrow k}$ respectively. Furthermore, we also  define the orthogonal projections $ \projc{k-1}{k},  \projc{k+1}{k}: \mathring\Omega^k \mapsto \mathring \Omega ^k$ from $\mathring\Omega^k$ into $\mathring \Omega^{k-1\rightarrow k}$ and $\mathring \Omega^{k+1\rightarrow k}$ respectively. Here, by the orthogonal projection, we mean with respect to the $\langle \cdot, \cdot \rangle$ product of $\Omega^k$ and $\mathring \Omega^k$ defined in \eqref{e.intern_product}. 
\end{definition}

\begin{remark}\label{r.counting_dimensions}
Thanks to Proposition \ref{p.basic_calculus}, we can compute the dimensions of $\Omega^{k-1\rightarrow k}$ and $\Omega^{k+1 \rightarrow k}$. To be more precise, let $\Lambda_j:=[-j,j]^4\cap \Z^4$ be our four dimensional graph. The dimensions of the projections of $\Omega^k$ are the following
\begin{table}[h!]
	\begin{tabular}{|l|l|l|l|l|}
		\hline
		Type & $\Omega^{0\rightarrow 1}$ & $\Omega^{2\rightarrow1}=\Omega^{1\rightarrow 2}$ & $\Omega^{3\rightarrow 2}= \Omega^{2\rightarrow 3}$ & $\Omega^{4\rightarrow 3}= \Omega^{3\rightarrow 4}$  \\ \hline
		Free& $(2j+1)^4-1$              & $(2j+1)^3(6j-1)+1$                       & $(2j)^3(6j+4)$                & $(2j)^4$                                              \\ \hline
		$0$  & $(2j-1)^4$                & $(2j-1)^3(6j+1)$                         & $(2j)^3(6j-4)+1$                   & $(2j)^4-1$                                           \\ \hline
	\end{tabular}
\end{table}

This table will help us compute the free-energy of the Gaussian-spin wave. The most important part of the table is to note that the dimension of $\Omega^{2\rightarrow 1}$ is of order $3 (2j)^4$, however the dimension of $\Omega^1$ is of order $4(2j)^4$. This explains the $3/4$ term in \eqref{e.d_free_energy}. 
\end{remark}

The following 
lemma gives a description of $\proj{k-1}{k}$ and $\proj{k+1}{k}$.

\begin{lemma}\label{l.proj}
	Let $\Lambda$ be a finite graph. Then, for all $1\leq k \leq n-1$
	\begin{align*}
	&\proj{k-1}{k} = \d \d^* \Delta^{-1},\ \ \  \projc{k-1}{k} = \d \d^* \mathring \Delta^{-1}, \\
	&  \proj{k+1}{k} = \d^* \d \Delta^{-1}, \ \ \  \proj{k+1}{k} = \mathring \d^* \d \mathring \Delta^{-1}.
	\end{align*}
\end{lemma}
\begin{proof}
	We prove the case for free-boundary condition, the zero-boundary condition is the same. 
	Recall that $\d$ and $\d^*$ commute with $\Delta^{-1}$. Thus, the image of the operators $\d \d^* \Delta^{-1} = \d  \Delta^{-1} \d^*$ and $\d^* \d \Delta^{-1} = \d^* \Delta^{-1} \d$ is $\Omega ^{k-1\rightarrow k}$ and $\Omega^{k+1 \rightarrow k}$ respectively.  Now, note that for any $f\in \Omega^k$ we have 
	\begin{align*}
	f = (\d \d^* + \d^* \d) \Delta^{-1} f = \d \d^* \Delta^{-1} f + \d^* \d \Delta^{-1}f.
	\end{align*}
	To conclude we just use that thanks to Proposition \ref{p.basic_calculus}, $\Omega^{k-1\rightarrow k}$ is perpendicular to $\Omega^{k+1 \rightarrow k}$.
\end{proof}

These orthogonal projections are useful to re-express the law of the gradient spin-wave from Definition \ref{d.SW} as follows.

\begin{proposition}\label{r.SW}
	The gradient spin-wave with free-boundary condition is the Gaussian process on $\Omega^{1\shortrightarrow 2}\subset \Omega^2$ which has density 
	\begin{align}\label{e.GSW_free_density}
	\frac
	{d \P^{\mathrm{GSW}}_{\beta,\Lambda}}
	{d \calL_{1\shortrightarrow 2}}
	(\varrho ) 
	&:= 
	\frac 1 {Z_{SpW}}   \exp\left( -\frac \beta 2   \sum_{f\in F(\Lambda)}  \varrho^2(f) \right)\propto \exp\left(-\frac{\beta}{2} \langle \varrho, \varrho \rangle \right),
	\end{align}
	where $\calL_{1\shortrightarrow 2}$ denotes the Lebesgue measure on $\Omega^{1\shortrightarrow 2}$. In other words, $(\langle \rho, f \rangle)_{f\in \Omega^2}$ is a centred Gaussian process with variance
	\begin{align*}
	\E^{GSW}\left[\langle \rho, f \rangle^2 \right] = \langle \proj{1}{2} f, \proj{1}{2} f\rangle. 
	\end{align*}
	
	For the case of the zero-boundary condition, the gradient spin wave is the Gaussian process on $\mathring \Omega^{1\shortrightarrow 2}$ with  density given by \eqref{e.GSW_free_density} where $\calL$ is replaced by $\mathring\calL$ the Lebesgue measure on $\mathring \Omega^{1\shortrightarrow 2}$.
	
	\begin{proof}
		This just follows from the fact that if $\varrho:= \d \phi$ is a GSW, where $\phi$ is a GFF, then $(\langle \rho, f \rangle)_{f\in \Omega^2}$ is the centred Gaussian process with variance given by
		\begin{align*}
		\E\left[ \langle \d \phi, f \rangle^2\right] &= \E\left[ \langle \phi, \d^*f \rangle^2\right]\\
		&= - \langle \d^* f, \d^* \Delta^{-1} f \rangle \\
		&= \langle f, \d \d^* \Delta^{-1} f\rangle\\
		&=\langle\proj{1}{2}f, \proj{1}{2} f\rangle.
		\end{align*}
The case of zero-boundary conditions follows using the same proof. 
	\end{proof}
\end{proposition}

\begin{remark}\label{r.Odecomp}
Such orthogonal decompositions for $k$-forms also exist on the infinite lattice $\Z^n$ but this requires some further care. First, one needs to add some integrability conditions in the definitions of $\Omega^{k-1\rightarrow k}$ and $\Omega^{k+1 \rightarrow k}$ as follows 
\begin{align}
	\label{e.k-1_k}&\Omega^{k-1\rightarrow k} = \d(\Omega^{k-1}) \cap \{f \text{a $k$-form}: \langle f, f\rangle <\infty \},\\
	\label{e.k+1_k}&\Omega^{k+1 \rightarrow k} = \d^*(\Omega^{k+1}) \cap \{f \text{a $k$-form}: \langle f, f\rangle <\infty \}.
	\end{align}
It is then a non-trivial exercise to check that the identities from Lemma \ref{l.proj} still hold at least when $n\geq 3$ and when applied to local functions for example. 
\end{remark}

\section{Decoupling between spin-wave and Coulomb gas}\label{s.decoupling}
Our goal in this section is twofold: 
\bnum
\item To provide a decoupling statement for Villain-$U(1)$ lattice gauge model. This will produce two independent random variables: a Gaussian spin-wave  (Definition \ref{d.SW})  and a Coulomb gas (see Definition \ref{d.Coulomb} below) which will live on the $3$-cells of $\Lambda\subset \Z^4$. 
\item Along the way, and in the same fashion as in \cite{GS2}, we will obtain a useful algorithm to sample the Coulomb gas out of the ``angular spins'' $\{\theta_e\}_{e\in \overrightarrow C^1(\Lambda)} \sim \P_\beta^{Vil}$. 
\enum 

We start with a definition of the Coulomb gas which arises in our present context. 
We also refer the reader to \cite{GS2} where Coulomb gases in dimension 2 are discussed in detail (including their behavior with respect to the rooting vertex). 

\begin{definition}\label{d.Coulomb}
(Coulomb gas) Let $\Lambda \subset \Z^4$ be a finite graph.
The Coulomb-gas associated to Villain-$U(1)$ lattice gauge theory on $\Lambda$ is the following probability measure on integer valued 3-forms: 
\begin{itemize}
	\item \textbf{Free boundary condition:} 
	\begin{align*}\label{}
	\FK{\beta,\Lambda, free}{\mathrm{Coulomb}}{\{q\}} \propto \exp\left (-\frac {\beta(2\pi)^2} 2 \<{q, (-\Delta)^{-1} q}\right ) 1_{\d q=0} \, \ \ \text{for any $q\in \Omega^3_\Z(\Lambda)$.}
	\end{align*}
	\item \textbf{Zero-boundary condition:}
	\begin{align*}\label{}
	\FK{\beta,\Lambda, 0}{\mathrm{Coulomb}}{\{\mathring q\}} \propto \exp\left (-\frac {\beta(2\pi)^2} 2 \<{\mathring q, (-\mathring \Delta)^{-1} \mathring q}\right )  1_{\d \mathring q=0}\,  \ \ \text{for any $\mathring q\in \mathring \Omega^3_\Z(\Lambda)$}\,.
	\end{align*}
\end{itemize}
\end{definition}

We state now the main proposition of this section.
\begin{proposition}\label{p.decouplingF}
	Let $(\theta, m)$ be a Villain $U(1)$-Lattice gauge theory with any boundary condition and define
	\begin{align*}
	&q:= \d m\\
	&\varrho:= \begin{cases}
	\d\theta + (2\pi)\proj{1}{2}m &  \text{ if $(\theta,m)$ has free-boundary condition}\\
		\d\theta + (2\pi) \projc{1}{2}m &  \text{ if $(\theta,m)$ has zero-boundary condition}
	\end{cases}
	\end{align*}
	Then
	\begin{enumerate}
		\item $q$ is independent of $\varrho$.
		\item $q$ has the law of a Coulomb gas on the $3$-forms with the same boundary condition as $(\theta,m)$.
		\item $\varrho$ is a gradient spin-wave on the $2$-forms with the same boundary condition as $(\theta,m)$.
	\end{enumerate}
\end{proposition}


\begin{proof} We focus on the free-boundary case. The proof for $0$-boundary condition follows the exact same lines.
	
Notice that using the notation $n_q$ from Definition \ref{d.n_q}, we have 
	\begin{align*}
	\varrho&:=\d \theta + (2\pi) \proj{1}{2} \mm\\
	&= \d \theta + (2\pi) \mm - (2\pi) \d^* \d \Delta^{-1} \mm\\
	&= \d \theta +(2\pi) \mm - (2\pi) \d^* \d \Delta^{-1} n_q\\
	&=\d \theta + (2\pi) \mm -(2\pi) n_\qq + (2\pi) \d \d^* \Delta^{-1} n_\qq\,.
	\end{align*}

	Let us take $W_1, W_2$ two continuous function from the $2$-forms to $\R$. 
	\begin{align}\label{e.obj_corr}
	\E_\beta^{Vil}\left[ W_1(\varrho)W_2(q)\right]&=\frac{1}{Z_\beta^{Vil}}\sum_{\substack{\mm }} \int_{[-\pi,\pi)^{C^1(\Lambda)}} W_1(\varrho) W_2(\qq) e^{-\frac{\beta}{2}\langle \d \theta +2\pi\mm, \d \theta + 2\pi\mm \rangle }d\theta.
	\end{align}
	Note that $\langle \d \theta +2\pi\mm, \d \theta + 2\pi\mm \rangle$ 	 is equal to
	\begin{align*}
	&\langle \d \theta +(2\pi)\proj{1}{2}\mm+(2\pi)\proj{3}{2}\mm, \d \theta + (2\pi)\proj{1}{2}\mm+(2\pi)\proj{3}{2}\mm \rangle\\
 &=\langle \varrho, \varrho \rangle + (2\pi)^2\langle\d^*\d\Delta^{-1}\mm ,\d^*\d\Delta^{-1}\mm \rangle \\
&= \langle \varrho, \varrho \rangle + (2\pi)^2\langle \qq, (-\Delta)^{-1} \qq \rangle.
	\end{align*}

	Recall now the bijection of Proposition \ref{p.bijection} and apply it to $\mm$. Note that in this bijection, we have  $\mm-n_q= \d \psi$. This implies that \eqref{e.obj_corr} is equal to
	\begin{align}
	\nonumber&\frac{1}{Z_{\beta}^{Vil}}\sum_{[\psi]}\sum_{\substack{ \d\qq=0} } \int_{[-\pi,\pi)^{C^1}}W_1(\varrho)W_2(\qq) e^{-\frac{\beta}{2}(\langle \varrho, \varrho \rangle + (2\pi)^2 \langle \qq, (-\Delta)^{-1}\qq\rangle)} d\theta\\
	&=\frac{1}{Z_{\beta}^{Vil}}\sum_{ \d\qq=0} \left (W_2(\qq)e^{-\frac{\beta(2\pi)^2}{2}\langle \qq, (-\Delta)^{-1}\qq\rangle }  \sum_{[\psi]} \int_{[-\pi,\pi)^{C^1}}W_1(\varrho) e^{-\frac{\beta}{2}\langle \varrho, \varrho \rangle } d\theta\right).\label{e.end_decoupling}
	\end{align}
	In the $2d$ Villain model, one may readily conclude as the sum $\sum_{\substack{[\psi]}}$ is nothing but a $\sum_{\substack{\psi}}$ where $\psi$ are rooted at some prescribed vertex. This is no longer the case as many integer-valued $1$-forms belong to the same $[\psi]$. The claim below allows us to overcome this degeneracy difficulty and thus concludes the proof of Proposition \ref{p.decouplingF}.
\end{proof}
	\begin{claim} \label{c.Lebesgue_1_2}
		Let us fix a $3$-form $\qq$ with $\d \qq=0$. We define the following measure on the $2$-forms
		\begin{align*}
		\mu(A):= \sum_{[\psi]}\int_{[-\pi,\pi)^{C^1}} \1_{A}(\d \theta + 2\pi\d \psi + 2\pi\d \d^* \Delta^{-1} n_\qq ) d\theta.
		\end{align*}
		Then, there exists a deterministic constant $\cons>0$ which depends neither on $\qq$ nor on $\beta$ (but depends on\footnote{Of course, it also changes when one changes the boundary condition.} $\Lambda$ ) such that
		\begin{align*}
		\mu= \cons \mathcal L _{1\shortrightarrow 2}.
		\end{align*}
		In other words, $\mu$ is a constant times the Lebesgue measure on $\Omega^{1\shortrightarrow 2}$.  
	\end{claim}
\begin{proof}
$ $ 
	\begin{enumerate}
		\item \textit{The measure is supported on a subset of $\Omega^{1\shortrightarrow 2}$.} This follows directly from the fact that $\d \theta +2\pi \d \psi + 2\pi\d \d^* \Delta^{-1} n_\qq $ always lives in $\Omega^{1\shortrightarrow 2}$.
		\item \textit{The measure of the unit ball is finite.} (Recall we work with a finite graph $\Lambda\subset \Z^4$ here). To do this, it is easier to work with the ball in the infinity norm. First, note that $\|\d \theta\|_\infty \leq 2\pi$. Now, we use the fact that for any $[\psi_1]\neq [\psi_2]$ the distance between $\d \psi_1$ and $\d \psi_2$ must be bigger than or equal to 1. As such for any 1-form $g$, there are finitely many equivalence classes $[\psi]$ such that $2\pi \d \psi$ is a (infinity norm) distance less than or equal to $4\pi$ from $\d g$. This concludes the proof of the fact.

		\item \textit{The measure is invariant under shifts in $\Omega^{1\shortrightarrow 2}$.} For this, let us take a (real-valued) $1$-form $h'$ and let us compute
		\begin{align}\label{e.shift_invariant}
		\mu(A-\d h')&= \int_{[-\pi,\pi)^{E}} \sum_{[\psi]} \1_{A}(\d \theta + 2\pi \d \psi +2\pi\d \d^* \Delta^{-1} n_\qq + \d h' ) d\theta,
		\end{align}
		Now, let us fix $\theta$ and define
		\begin{align*}
		&\theta'= \theta + h' \mod 2\pi\\
		&\psi'= \frac{1}{2\pi}(\theta + h' - \theta ').
		\end{align*}
		Noting that $\psi'$ takes always its values on the integers, we see that
		\begin{align*}
		\mu(A-\d h')&= \int_{[-\pi,\pi)^{E}} \sum_{[\psi]} \1_{A}(\d \theta ' + 2\pi \d (\psi+\psi') + 2\pi\d \d^* \Delta^{-1} n_\qq ) d\theta,\\
		&=\int_{[-\pi,\pi)^{E}} \sum_{[\psi]} \1_{A}(\d \theta ' + 2\pi \d\psi + 2\pi\d \d^* \Delta^{-1} n_\qq ) d\theta,\\
		&=\sum_{[\psi]}\int_{[-\pi,\pi)^{E}}  \1_{A}(\d \theta + 2\pi \d\psi+ 2\pi\d \d^* \Delta^{-1} n_\qq ) d\theta = \mu(A).
		\end{align*}
\item {\em  The measure of the unit ball does not depend on $q$.} This follows from the above proof applied to $h' = 2\pi \d^* \Delta^{-1} n_\qq$. 		
	\end{enumerate}
\end{proof}


\begin{remark}\label{r.ImpSpW} 
	In the proof of Proposition \ref{p.decouplingF}, the key step was to recognize the Lebesgue measure $\calL_{1\to 2}$ on $\Omega^{1\shortrightarrow 2}$ as this allowed us to make the link with our gradient spin-wave as defined in Proposition  \ref{r.SW}. This proof strongly relies on the bijection using equivalent classes in Proposition \ref{p.bijection}. Note that one may have tried using the following two other natural ways to deal with the above quotienting issue.
		\bnum
		\item A  cut-off procedure, which would correspond to summing over all  integer $1$-forms $\psi$ with values in $[-K,K]^{C^1(\Lambda)}$ and then letting $K\to \infty$. The difficulty here is to handle the "boundary issues" near $\p [-K,K]^{C^1(\Lambda)}$. 
		\item Adding a small mass $\mathbf{m}$ to the spin-wave and letting $\mathbf{m}\to 0$. This is related to the approach followed in \cite{FSrestoration}. The disadvantage of this approach here is that it also affects the gradient spin-wave measure, i.e.  $d\varphi^{\mathbf{m}} \neq d\varphi$ and as such was less convenient for our present decoupling. 
		\enum
As shown above, we instead relied on the characterization of the Lebesgue measure on $\d \Omega^1(\Lambda)$ by its invariance under shifts. 
	\end{remark}

Let us now prove a simple corollary which will allow us to compute the Fourier transform of $\d \theta  + 2\pi m$. 
\begin{corollary}\label{c.correlation}
Let $(\theta,m)$ be a Villain $U(1)$ lattice gauge coupling in a finite graph $\Lambda$ (with any boundary condition) and $\varrho$ be a spin-wave at inverse temperature $\beta$ with the same boundary condition. Then, we have that for any $2$-form $f$
\begin{align*}
\E_{\beta}^{Vil}\left[e^{i\langle \d \theta + 2\pi m, f\rangle} \right] = \E_{\beta}^{GSW}\left[e^{i\langle \varrho ,  f \rangle} \right] \E_{\beta}^{Vil}\left[e^{i2\pi \langle m,\proj{3}{2} f \rangle} \right]. 
\end{align*}
In particular, as $m$ takes values in the integers, we have that if $f$ is a $2$-form taking only integer values 
\begin{align}\label{e.Fourier_decoupling_integer_f}
\E_{\beta}^{Vil}\left[e^{i\langle \d \theta, f\rangle} \right] & = \E_{\beta}^{GSW}\left[e^{i\langle \varrho , f \rangle} \right] \E_{\beta}^{Vil}\left[e^{i2\pi \langle m,\proj{3}{2} f \rangle} \right] \\
& =  \E_{\beta}^{GSW}\left[e^{i\langle \varrho , f \rangle} \right] \E_{\beta}^{Coul}\left[e^{i2\pi \langle \d^* \Delta^{-1} q, f \rangle} \right] \,,\nonumber
\end{align}	
(where in the last equality, $\d^*$ and $\Delta$ should be replaced by $\mathring \d^*$ and $\mathring \Delta$ in case of zero-boundary conditions for the Villain lattice gauge theory on $\Lambda$). 
\end{corollary} 

\begin{proof}
	The proof for both boundary condition is analogous and straighforward. We do it for the free-boundary condition.
	
Let us work with the same notation as that of Proposition \ref{p.decouplingF}. We note that
\begin{align*}
\d \theta + 2\pi m &= \varrho + 2\pi m - (2\pi)\proj{1}{2}m \\
&= \varrho + (2\pi) \proj{3}{2} m \\
&= \varrho + (2\pi) \d^* \Delta^{-1} q.
\end{align*}
Thus, as $q$ is independent of $\varrho$, we have that
\begin{align*}
\E_{\beta}^{Vil}\left[e^{i\langle \d \theta + 2\pi m, f\rangle} \right]&= \E_{\beta}^{GSW}\left[e^{i\langle \varrho , f \rangle} \right] \E_{\beta}^{Coul}\left[e^{i2\pi \langle \d^* \Delta^{-1} q,f \rangle} \right] \\
&= \E_{\beta}^{GSW}\left[e^{i\langle \varrho ,f \rangle} \right] \E_{\beta}^{Vil}\left[e^{i2\pi \langle m,\proj{3}{2} f \rangle} \right],
\end{align*}
where in the last line we used that
\begin{align*}
 \d^* \Delta^{-1} q = \d^* \d \Delta^{-1} m = \proj{3}{2} m.
\end{align*} 
\end{proof}
%
%
%
%
%
%

\section{Energy of a Wilson loop}\label{s.energy}
The goal in this section is to analyze the law of the random variable $\langle \varrho, \1_R\rangle$ where  $R$ is a two-dimensional rectangle in $\Z^4$ and $\varrho$ is a gradient spin-wave (Definition \ref{d.SW} and Proposition \ref{r.SW}). If the boundary of $R$ is given by the loop $\gamma$, i.e. if $\p R =\gamma$, this is also the law of $\langle \phi, \d^*\1_R\rangle = \langle \phi, \1_{\gamma}\rangle$.

Note that $\langle \varrho, \1_R \rangle$ is a centred Gaussian random variable with variance
\begin{align*}
\| \proj{1}{2} \1_R \|^2 = \langle \d \d^* \Delta^{-1} 1_R, \d \d^* \Delta^{-1} 1_R \rangle= \langle \d \d^* \Delta^{-1} 1_R, 1_R \rangle = - \langle \Delta^{-1} \1_{\gamma}, \1_{\gamma} \rangle.
\end{align*}
Here, depending on which graph we are working on, we use either the Laplacian for the given boundary condition or for the infinite volume case.

\begin{proposition}\label{p.energy_wilson}
	Assume that $R$ is a rectangle with height $H$ and length $L$ in a graph $\Lambda$, we have that as long as $ L^{3/4} \leq H \leq L^{4/3}$ and the distance between $R$ and the boundary of $\Lambda$ is bigger than $L$, then 
	\begin{align*}
\| \proj{1}{2} \1_R \|^2 =  C_{GFF} \, 2(L+H)+ O(L^{2/3}),
	\end{align*}
	where $C_{GFF}$ is the constant defined in \eqref{e.cGFF}.
	In particular, if $\varrho$ is gradient spin-wave with either boundary condition, we have that
	\begin{equation}
	\E_\beta^{GSW}\left[e^{i\langle \varrho, \1_R \rangle} \right] =e^{-\frac{C_{GFF}}{\beta}(L+H +O(L^{2/3}))}
	\end{equation}
\end{proposition}
\begin{proof}
To simplify the proof, we work with $\Lambda= \Z^4$, and at the end we will  discuss how to extend it to finite graphs whose boundary is at sufficient distance from $R$.

Let us separate $\gamma$ in its horizontal and vertical part $\gamma^h$, $\gamma^v$
so that 
\begin{align*}
\| \proj{1}{2} \1_R \|^2 & =\langle \Delta^{-1} \1_{\gamma}, \1_{\gamma} \rangle \\
&  = \sum_{e \in \gamma^h}\Delta^{-1} \1_{\gamma^h}(e) + \sum_{e \in \gamma ^v}\Delta^{-1} \1_{\gamma^v}(e)\,.
\end{align*}

%

Without loss of generality we need to show that
\begin{align}\label{e.hminus}
\sum_{e \in \gamma^{h,-}}\Delta^{-1} \1_{\gamma^h}(e) = - C_{GFF}L + O(L^{1/2}),
\end{align}
where $\gamma^{h,-}$ is the lower part of $\gamma^{h}$.
Note that thanks to Corollary \ref{c.Green_1}, only horizontal edges contribute to $\Delta^{-1} 1_{\gamma}(e)$ when $e$ is horizontal. This gives us
\begin{align*}
-\Delta^{-1} \1_{\gamma^h}(e)&=\sum_{e'\in\gamma^{h}} -(\Delta^{-1})(e,e')\\
&= \sum_{e'\in\gamma^{h,-}} -(\Delta^{-1})(e,e') +\sum_{e'\in\gamma^{h,+}} -(\Delta^{-1})(e,e')\\
& \leq \sum_{e' \in \gamma^{h,-}_\infty} -(\Delta^{-1})(0, e') +  \sum_{e'\in\gamma^{h,+}_\infty} -(\Delta^{-1})(e,e')
\end{align*}
where $\gamma^{h,-}_\infty$ is the infinite line that passes through $\gamma^{h,-}$ and $\gamma^{h,+}_\infty$ is the infinite line that passes through $\gamma^{h,+}$. The term on the right is the expected number of visits of an infinite line $\Z^4$ for a SRW starting at distance $H$ from it. Standard Green function estimates  in $\Z^4$  (Proposition \ref{p.Green_vertices}) show that this is of order $H^{-1}$ which is itself less than $L^{-1/2}$.

To obtain a lower bound, note that for all edges $e$ that are at distance  bigger than $\sqrt{L}
$ of the corner of the rectangle, then we have 
\begin{align*}
 C_{GFF}-\sum_{e'\in\gamma^{h,-}} -(\Delta^{-1})(e',e) \leq O(1) \sum_{k\geq \sqrt{L}} k^{-2} \leq O(1) L^{-1/2}\,.
\end{align*}
This ends the proof of the etimate~\eqref{e.hminus}. The vertical sides are handled the same way (except we need to exclude points at distance $L^{2/3}$ from the corners) 
which thus concludes the proof in the case of the infinite volume limit. 

The result in the case of a finite graph is done by noting that 
\begin{align*}
&|(\Delta_{\Z^4}^{-1})(e',e)-(\Delta_{\Lambda}^{-1})(e',e)| =O\left (\frac{1}{L^2}\right )\\
&|(\mathring \Delta_{\Z^4}^{-1})(e',e)-(\mathring \Delta_{\Lambda}^{-1})(e',e)| =O\left (\frac{1}{L^2}\right )
\end{align*}
for all $e, e' \in \gamma$. This gives at most a correction of order $O(1) (L^{4/3})^2*L^{-2}=O(L^{2/3})$ as desired. 

\end{proof}

We conclude this section by analyzing how {\em well-spread}  the energy of $\proj{1}{2} \1_R$ is.
\begin{proposition}\label{p.spread_energy}
	Let us work in the context of Proposition \ref{p.energy_wilson}. Then, for all $b>0$ there exists a constant $C=C(b)>0$ such that
	\begin{align*}
\|(\proj{1}{2} \1_R) \1_{|\proj{1}{2} \1_R|<b}\|^2&=\langle(\proj{1}{2} \1_R) \1_{|\proj{1}{2} \1_R|<b},(\proj{1}{2} \1_R) \1_{|\proj{1}{2} \1_R|<b} \rangle\\
&\geq C \langle \proj{1}{2} \1_R, \proj{1}{2} \1_R \rangle.
	\end{align*}
	The same is true for zero-boundary condition, i.e.,
	\begin{align*}
\|(\projc{1}{2} \1_R) \1_{|\projc{1}{2} \1_R|<b}\|^2\geq C\langle \projc{1}{2} \1_R, \projc{1}{2} \1_R \rangle.
	\end{align*}
\end{proposition}
\begin{remark}
Note that $C(b)$ increases as $b$ increases.
\end{remark}
\begin{proof}
	The proof is equivalent for either boundary condition. We will do the proof for free-boundary condition.
	
	To prove this proposition, we will study what the value of $\proj{1}{2} \1_R$ is at a given distance $k$ from the boundary (the smaller $b$ is, the larger $k$ will be). More precisely, we will show that for any 2-cell $f\in R$ that is at distance $k$ of one side of the rectangle and at distance at least $\sqrt{L}$ from the other sides,  
	\begin{align}\label{e.epsilon(k)}
	|\proj{1}{2} \1_R(f)|^2 = \epsilon(k) - O(k^{-3}) - O(L^{-{1/2}})\,,
	\end{align}
uniformly on all the graphs $\Lambda$ whose boundary is at distance \dgreen{$L$} of $R$. Here $\epsilon: \Z\mapsto (0,1)$ is a function going to $0$ as $k\nearrow \infty$ which satisfies $\eps(k)\geq \Omega(k^{-2})$, it is defined below in~\eqref{e.epsk}. It is clear that \eqref{e.epsilon(k)} is enough as  the energy associated to points at distance $k$ of the side of the rectangle  will already have an energy larger than $\epsilon(k)(L+H) -O(L^{-1/2})$.
	
	We now show \eqref{e.epsilon(k)}. Let us first work with $\Lambda= \Z^4$,  the result for finite $\Lambda$ will follow again by using Lemma \ref{l.approximation2}. In this case, we have the identity\footnote{As pointed out in Remark \ref{r.Odecomp}, this identity in the infinite volume case requires some care. The reader may either prove this identity for local functions $f$ (such as $f=\1_R$ here) or otherwise make all the computations in this proof in a very large cube $\hat \Lambda$ and then only at the level of the obtained estimates pass to the infinite volume limit $\hat \Lambda \nearrow \Z^n$. This other option is feasible as both the gradient spin-wave and the $U(1)$-Villain lattice gauge theory have well-defined infinite volume limits (Proposition \ref{pr.IVL}).}
	\[\proj{1}{2} \1_R = \d \Delta^{-1} \1_{\gamma}.\]

	To compute $|\proj{1}{2} \1_R(f)|^2$, we use Corollary \ref{c.Green_1}. Assume that $f\in R$ is at distance $k$ of the horizontal line of $R$ and at distance greater than or equal to $\sqrt L$ of the vertical lines. Then, write $e_1$ and $e_3$ the two horizontal edges of $f$ and $e_2$ and $e_4$ the two vertical edges. We chose $e_1$ and $e_3$ being parallel with the same direction, the same as $e_2$ and $e_4$. Thus,
	\begin{align*}
	\d \Delta^{-1}\1_{\gamma}(f)& = \Delta^{-1}\1_{\gamma}(e_1)-\Delta^{-1}\1_{\gamma}(e_3) +\Delta^{-1}\1_{\gamma}(e_2)-\Delta^{-1}\1_{\gamma}(e_4) \\
	&=  \Delta^{-1}\1_{\gamma}(e_1)-\Delta^{-1}\1_{\gamma}(e_3) +O\left (\frac{1}{\sqrt L} \right )
	\end{align*}
	were the last equation is just done by using the fact that the inverse Laplacian on 1-forms for vertical edges $e_2$ and $e_4$ needs to go at distance at least $\sqrt{L}$ to find lines of other vertical edges. We now, use Proposition \ref{p.Green_vertices} to estimate that 
	\begin{align*}
	\Delta^{-1}\1_{\gamma}(e_1)-\Delta^{-1}\1_{\gamma}(e_3) & =  \sum_{e \in \gamma_h} C_4 (4-2)  \frac{1}{\|e_1-e\|^{4-1}} + O\left (\frac{1}{\|e_1-e\|^{4}}\right ) \\
	& = 2 C_4 \sum_{e\in \gamma_{h,\infty}} \frac{1}{\|e_1-e\|^3} +  O(k^{-3}) + O(L^{-1})\,,
	\end{align*}
where $\gamma_{h,\infty}$ is the infinite continuation of the horizontal line $\gamma_h$ closer to $e$ and where the correction term $O(L^{-1})$ arises from the difference between $\gamma_{h,\infty}$ and $\gamma_{h}$ while the correction term $O(k^{-3})$ arises from the sum of correction terms $O(\|e_1-e\|^{-4})$.   This  concludes the proof by defining
\begin{align}\label{e.epsk}
\epsilon(k)&:= 2 C_4 \sum_{e\in \gamma_{h,\infty}} \frac{1}{\|e_1-e\|^3}
\end{align}
which is indeed larger than $\Omega(k^{-2})$. 
\end{proof}

\section{Corollaries of the spin-wave decoupling}\label{s.coro}

In this short section, we give several easy corollaries of the spin-wave decoupling property (Proposition \ref{p.decouplingF}). Except in the last subsection (on a sampling algorithm  for Coulomb gas in dimensions $n\geq 3$), they do not require the novelty of the joint coupling $(\theta,m)$ which will be needed for our main statement on the spin-wave improvement. Yet, they allow us to recover at once different celebrated results from the literature and they also shed some light on the connection with the recent results proved for discrete lattice gauge theory in \cite{Sourav,malin,cao}. 

\subsection{Decoupling in terms of partition functions.}
We start with the following straightforward corollary of the Coulomb/spin-wave decoupling for the partition functions of the models. 

\begin{corollary}\label{l.decoupling_partition_function}
	We have the following decoupling result for the partition function of the $U(1)$-lattice gauge theory on the graph $\Lambda_j:=[-j,j]^4\cap \Z^4$
	\begin{align*}
	Z^{Vil}_{\beta,\Lambda_j}= \cons Z^{GSW}_{\beta,\Lambda_j}Z^{Coul}_{\beta,\Lambda_j}\,,
	\end{align*}
	where the constant $\cons$ only depends on the graph $\Lambda$ and its boundary conditions.
	(Here, all the partition function have the same boundary condition: either zero or free boundary condition). 
\end{corollary}
\begin{proof}
	Note that the proof of Proposition \ref{p.decouplingF} implies that for any $W_1:\Omega^{2}\to \R$ and  $W_2:\Omega^3\to \R$
	\begin{align*}
&\E_\beta^{Vil}\left[ W_1(\varrho)W_2(q)\right]\\
&\hspace{0.1\textwidth}=\frac{1}{Z_{\beta}^{Vil}}\sum_{ \d\qq=0} \left (W_2(\qq)e^{-\frac{\beta(2\pi)^2}{2}\langle \qq, (-\Delta)^{-1}\qq\rangle }  \sum_{[\psi]} \int_{[-\pi,\pi)^{C^1}}W_1(\varrho) e^{-\frac{\beta}{2}\langle \varrho, \varrho \rangle } d\theta\right).
\end{align*}
The result follows from taking $W_1=1$ and $W_2=1$ and using Claim \ref{c.Lebesgue_1_2}.
\end{proof}

\subsection{McBryan-Spencer spin-wave bound.}
By combining Proposition \ref{p.energy_wilson} on the energy of a Wilson loop with the decoupling Proposition \ref{p.decouplingF}, we obtain the following analog of  McBryan-Spencer's upper bound \cite{McBryanSpencer} in the present context of $4D$ $U(1)$-lattice gauge theory. 
\begin{corollary}\label{c.Mc}
	In the special setting of the $U(1)$-lattice gauge theory with Villain interaction, we recover the general result from \cite{SimonYaffe} which states that for any gauge group, 
	\begin{align*}\label{}
	|\EFK{\beta}{}{W_\gamma}| \leq \exp(-c(\beta) |\gamma|).
	\end{align*}
	Furthermore, we obtain that $c(\beta)\geq C_{GFF}/(2\beta)$ (where $C_{GFF}$ was defined in~\eqref{e.cGFF}). This corresponds in this $4D$  $U(1)$-lattice gauge theory setting to the McBryan-Spencer bound \cite{McBryanSpencer}. 
\end{corollary}

\subsection{Quark trapping in $3d$ $U(1)$ gauge theory.} 
The energy of a Wilson loop may also be easily analyzed in $d=3$. In this case, Proposition \ref{p.energy_wilson} translates easily as follows: if $\gamma$ is a rectangular loop in $\Z^3$ of side-lengths $H \leq L$ which is sufficiently far from $\p \Lambda$, then there exists a constant $C>0$ such that its energy $\| \proj{1}{2} \1_R \|^2$ satsifies 
\begin{align*}\label{}
\| \proj{1}{2} \1_R \|^2 \geq C L \log H\,.
\end{align*}
This in turn implies the following result (for the Villain interaction) which is due to Glimm-Jaffe \cite{glimm1977quark}.

\begin{corollary}[\cite{glimm1977quark}]\label{c.jaffe}
Consider any infinite-volume limit in $\Z^3$ of Villain $U(1)$-lattice gauge theory equipped with either free or Dirichlet boundary conditions.  
There exists $C>0$ such that for any $\beta>0$ and any rectangular loop $\gamma$ with 
side-lengths $H \leq L$,
\begin{align*}\label{}
|\EFK{\beta}{}{W_\gamma}| \leq \exp\left (-\frac{C}{2\beta} L \log H\right ).
\end{align*}
\end{corollary}
As observed in \cite{glimm1977quark}, this result is physically very much relevant as it indicates a phenomenon of {\em quark trapping} at all inverse temperatures $\beta>0$ for $U(1)$-lattice gauge theory. This was latter greatly improved (still in the case of $U(1)$-lattice gauge theory) by Göpfert-Mack in \cite{GM82}.

\subsection{Computation of Wilson's observables in the non-degenerate regime.}
Our next corollary is the following analog of the main statements proved in the series of works \cite{Sourav,malin,cao} for discrete gauge groups. It is a direct consequence of the main estimates from \cite{FSrestoration} combined with Proposition \ref{p.decouplingF}. 

\begin{corollary}\label{c.sourav}
As $\beta \to \infty$, one can compute exactly (i.e. up to $o(1)$-correction) the Wilson loop observables of large and sufficiently rectangular loops $\gamma$ which are sufficiently far from the boundary (in the sense of Theorem \ref{th.main_finite} below) and whose perimeter $|\gamma|$ satisfies as $\beta \to \infty$, $|\gamma|\asymp \beta$. Wilson loop observables for such loops $\gamma$ satisfy as $\beta\to \infty$  
\begin{align*}\label{}
\EFK{\beta}{}{W_\gamma} = e^{- \frac{C_{GFF}}{2\beta}  |\gamma| }+o_{\avelio{\beta}}(1)\,,
\end{align*}
where we recall that $C_{GFF}$ was defined in~\eqref{e.cGFF}. This result also holds for the case of infinite volume (subsequential)-limits on $\Z^4$. 
\end{corollary} 
As discussed in Subsection \ref{ss.links}, this Corollary is the analog of the results in \cite{Sourav,malin,cao} except loops $\gamma$ need to be much smaller: namely of size $\asymp \beta$ in our case when $G=U(1)$ as opposed to $|\gamma| \asymp e^{12\beta}$ in the case where $G=\Z_2$ (\cite{Sourav}).  Also when $G=U(1)$ such a result which is asymptotic in $\beta$ does not see the effect of vortices which would start being relevant for loops of size $\asymp \beta$ only at the next orders in $\beta$. 
Let us also highlight that the results in \cite{Sourav,malin,cao} are more quantitative in the geometry of the loop $\gamma$ (in particular in its number of corners). Here for simplicity, we sticked to the flat case with four corners.

\smallskip
\noindent
{\em Proof of Corollary \ref{c.sourav}.}
This is a direct consequence of the following facts.
\bnum
\item The estimate~\eqref{e.WL} which is one of the main estimate proved in \cite{FSrestoration}. 
\item Ginibre's inequalities \cite{Ginibre}, as used in \cite{FSrestoration} which imply that for all loops $\gamma$, $\EFK{\beta}{}{W_\gamma} \geq 0$
\item Proposition \ref{p.decouplingF} which, as in Corollary \ref{c.Mc} and McBryan-Spencer estimate provides the easier upper bound. 
\enum
\qed

\subsection{An efficient sampling algorithm for Coulomb gas in dimensions $n\geq 3$.}\label{ss.algo}

In \cite[Section 4]{GS2}, we introduced a sampling dynamics for the $2d$ Coulomb gas with local updates (despite the long range interactions in $\<{q,(-\Delta)^{-1} q}$). This algorithm used the decoupling of the Villain model into a GFF and Coulomb gas in $2d$. As explained in \cite[Section 4]{GS2}, an extension of this algorithm to higher dimensions $n\geq 3$ would require a suitable decoupling lemma for a Villain model defined on the $n-2$ cells. This is exactly what the analysis carried out in Section \ref{s.decoupling} allows us to do. We state it as a Corollary, again of Proposition \ref{p.decouplingF}.

\begin{corollary}
Fix  $n\geq 2$, $\beta>0$ and a box $\Lambda=\Lambda_j=[-j,j]^n \subset \Z^n$. The classical $n$-dimensional Coulomb gas\footnote{as opposed to the Coulomb gas in this paper which lives on the $3$-forms of $\Z^4$ and which is naturally associated to $U(1)$ gauge theory.} is the random $n$-form $q$ whose law is given by
\begin{align*}\label{}
\FK{\beta,\Lambda}{Coul}{q} \propto \exp(-\frac \beta 2 \<{q, (-\Delta)^{-1}q})\,,
\end{align*}
and where $\Lambda$ is equipped with either free on Dirichlet boundary conditions (see \cite{GS2} for the proper meaning of free boundary conditions when dealing with $n$-forms, which was never needed in this paper). 

One can implement a dynamics with local updates in order to sample this $n$-dimensional Coulomb gas. 
\end{corollary}

\noindent
{\em Proof.}
Following the algorithm outlined in \cite{GS2} in the $2d$ case, one can introduce the following sampling algorithm.
\bnum
\item Fix  $\Lambda\subset \Z^n$ a finite box with either free or Dirichlet boundary conditions.
\item The first main step is to sample a version of a Villain model $\theta$ which now lives on the $n-2$-forms of $\Lambda$. Exactly as in $2d$, it is defined using the following Gibbs measure
\begin{align*}\label{}
\FK{\beta}{}{\theta} \propto \prod_{c \in C^{n-1}(\Lambda) } \sum_m \exp(-\frac \beta 2 (\d \theta(c) + 2\pi m)^2)\,, 
\end{align*}
where recall from Subsection \ref{ss.reminder} that $C^{n-1}(\Lambda)\subset \overrightarrow C^{n-1}(\Lambda)$ is the set of oriented $n-1$ cells of $\Lambda$. 
 
 As in $2d$, one can run  a local MCMC chain in order to sample this Villain model at inverse temperature $\beta$.  This should require at most $O(|\Lambda|^{\eta(\beta)})$ steps. (See \cite{GS2} for a discussion). 

\item Once the Villain configuration has reached equilibrium, sample the random $n-1$-form $m$ conditionally independently on each positively oriented $n-1$ cell of $\Lambda$ as follows
\begin{align*}
		\P^{\beta}_{Vil}(dm \mid \theta )
		& \propto \prod_{c \in C^{m-1}(\Lambda)} \exp\left(-\frac{(2\pi)^2\beta}{2}\Big(\frac{\d \theta(c)}{2\pi} +  m(c)\Big)^2 \right) \delta_\Z(dm)\,. 
		\end{align*}
		
\item Finally obtain the $n$-dimensional Coulomb gas using $q=\d m$.
\enum 

The reason why this algorithm produces the desired law is due to a straightforward extension of the decoupling Proposition \ref{p.decouplingF} to the setting of a Villain model on the $n-2$ forms. {\em (N.B. This coincide with $U(1)$ lattice gauge theory only when $n=3$)}.  This was not available in \cite{GS2} as arguments such as Claim \ref{c.Lebesgue_1_2} were not needed for the decoupling result in $2d$. \qed

\section{Proof of the spin-wave improvement}\label{s.proof}
\subsection{Bound on the Fourier transform of $m$.}
Now that we introduced the correct framework, this subsection will follow the same analysis as in \cite{GS2}. The key step to obtain bounds on the Fourier transform of a Coulomb gas is the following lemma. The proof closely resembles that of Lemma 7.1 of \cite{GS2}.
\begin{lemma}\label{l.Fourier_m}
	Take $(\theta,m)$ a Villain $U(1)$ lattice gauge theory in $\Lambda$ at inverse temperature $\beta$ with either boundary condition and a $2$-form $h\in \Omega^2(\Lambda)$. For $b>0$ define 
\begin{equation}
h^{b}(w)= h(w) \1_{|h|<b}
\end{equation}
	

For all $\beta>0$ there exists $0<\tilde b(\beta)<1$ (which can be chosen to be increasing in $\beta$) such that for all $b<\tilde  b(\beta)$
\begin{equation*}
\E^{Vil}_{\beta}\left[e^{i\langle m,h \rangle} \right]\leq e^{-\frac{(1-bK_\beta)}{2} \inf_a \var^{IG}(a,(2\pi)^2\beta) \langle h^b, h^b \rangle}.
\end{equation*}
(Recall the definition of the constant $K_\beta$ in ~\eqref{e.K_beta}). 
\end{lemma}
\begin{proof}
	We start by separating the characteristic function conditioning on $\theta$
	\begin{align*}
	\E^{Vil}_{\beta}\left[ e^{i\langle m,h\rangle }\right]&=\E^{Vil}_{\beta} \left[ e^{i\langle \E\left[ m\mid \theta \right],h\rangle }\E\left[e^{i\langle m-\E\left[m\mid \theta \right] ,h\rangle } \mid \theta \right] \right]\\
	&=\E^{Vil}_{\beta} \left[ e^{i\langle \E\left[ m\mid \theta \right],h\rangle }\prod_{e\in C^1}\E\left[e^{i (m(e)-\E\left[m(e)\mid \theta \right]) h(e)} \mid \theta \right] \right]
	\end{align*}
	
	We now use that for all $x\in \R$
	\begin{align*}
	&|\sin(x)-x|\leq \frac{|x|^3}{3!}\\
	& \left |\cos(x)-1+\frac{x^2}{2}\right |\leq \frac{|x|^3}{3!},
	\end{align*}
	to see that $|\E_{\beta}^{Vil}\left[e^{i\langle m, h\rangle} \right]|$ is smaller than or equal to (using the notations from Subsection \ref{ss.IVG})
	\begin{align*}
	&  \E_{\beta}^{Vil}\left[ \prod_{e\in C^1}\left |\E\left[e^{i (m(e)-\E\left[m(e)\mid \theta \right]) h^b(e)} \mid \theta \right]\right|  \right]\\
	&\leq \E_{\beta}^{Vil}\left[\prod_{e\in C^1}\left( 1-\frac{1}{2}(h^b(e))^2\var^{IG}(\d\theta(e),(2\pi)^2\beta) +\frac{|h^b(e)|^3}{3}T^{IG}(\d\theta(e),(2\pi)^2\beta)\right )  \right] 
	\end{align*}
as long as for every edge $e\in C^1$, 
	\begin{equation}
	\frac{1}{2}(h^b(e))^2\var^{IG}(\d\theta(e),(2\pi)^2\beta) \leq 1.
	\end{equation}
	Let us note that this holds as long as $b< \tilde b(\beta)$, where $\beta \mapsto  \tilde b(\beta)$  is the following increasing function of $\beta$
	\begin{align*}
	\tilde b (\beta) :=\min \left \{ \inf_{a,\hat \beta \geq \beta }\frac{1}{\sqrt{\var^{IG}(a,(2\pi)^2\hat \beta)}},  K_{\beta}^{-1} \right \},
	\end{align*}
	
	We now use that $\log(1+x)\leq x$ as long as $x>-1$. This implies that $|\E^{Vil}_{\beta}\left[e^{i\langle m, h\rangle} \right]|$ is upper bounded by
	\begin{align*}
	&\E^{Vil}_\beta\left[\prod_{e\in E}e^{ -\frac{1}{2}(h^b(e))^2\var^{IG}(\d\theta(e),(2\pi)^2\beta^{-1})+\frac{|h^b(e)|^3}{3!}T^{IG}(\d\theta(e),(2\pi)^2\beta^{-1})}  \right]\\
	&\hspace{0.3\textwidth}\leq e^{ -\min_a \frac{1}{2}\var^{IG}(a,(2\pi)^2\beta^{-1})\langle h^b,h^b \rangle +\frac{b}{3!} T^{IG}(a,(2\pi)^2\beta^{-1})\langle h^b,h^b\rangle}.
	\end{align*}

\end{proof}

\subsection{Proof of the main theorem.}
We will now prove our improved spin-wave estimate for the Wilson loop observable.  The theorem stated below is a more precise version of Theorem \ref{th.main}. 

\begin{theorem}\label{th.main_finite}
Let $\Lambda\subset \Z^4$ be the infinite lattice or a finite cube,  equipped either with free or Dirichlet boundary conditions. Assume that $R$ is a rectangle in $\Lambda$ with height $H$ and length $L$ so that $ L^{3/4} \leq H \leq L^{4/3}$. In case where $\Lambda$ is finite we assume furthermore that the distance between $R$ and $\p \Lambda$ is bigger than $L$. Then, there exists a constant $K$ s.t. for all $\beta\geq 1$, 
	\begin{align*}
	\E^{\Vil}_\beta\left[e^{\langle \d \theta, \1_{R} \rangle} \right]&\leq  e^{-\frac{1+ K\, M(\beta)}{2\beta} \langle \proj{1}{2} \1_R, \proj{1}{2} \1_R \rangle}\\
	&\leq  e^{- \frac{1+ K\, M(\beta)}{\beta}\, C_{GFF}  (L+H+O(L^{2/3}))}\,.
	\end{align*}
\end{theorem}
\begin{proof}
	We start by using Corollary \ref{c.correlation}, in particular \eqref{e.Fourier_decoupling_integer_f}
	\begin{align}
	\E_{\beta}^{Vil}\left[e^{i\langle \d \theta, \1_R \rangle} \right] &= \E_{\beta}^{GSW}\left[e^{i\langle \varrho , \1_R \rangle} \right] \E_{\beta}^{Vil}\left[e^{i2\pi \langle m,\proj{3}{2} \1_R \rangle} \right] \nonumber \\
	\label{e.end}
	&\leq\E_{\beta}^{GSW}\left[e^{i\langle \varrho , \1_R \rangle} \right] \E_{\beta}^{Vil}\left[e^{i2\pi \langle m,\proj{1}{2} \1_R \rangle} \right],
	\end{align}
	where we used that $\proj{3}{2} \1_R = \1_R - \proj{1}{2}\1_R$ and $\1_R \in \Z^{C^2}$. Then, taking a deterministic $b< \bar b ( 1)$, 
we can use Lemma \ref{l.Fourier_m} together with Proposition \ref{p.spread_energy}  to see that
	\begin{align}
	\E_{\beta}^{Vil}\left[e^{i2\pi \langle m,\proj{1}{2} \1_R \rangle} \right]
	&\leq e^{- \frac{(1-bK_\beta)}{2} \inf_a \var^{IG}(a,(2\pi)^2\beta)
	 \| \proj{1}{2} \1_R \1_{|\proj{1}{2} \1_R|<b}\|^2 }  \nonumber\\
	 & \leq  e^{- \frac{(1-bK_\beta)}{2} \inf_a \var^{IG}(a,(2\pi)^2\beta) 
	 C(b) \langle \proj{1}{2} \1_R, \proj{1}{2} \1_R \rangle} \nonumber \\
	  & \leq  e^{- \frac{1}{4} \inf_a \var^{IG}(a,(2\pi)^2\beta) 
	 C(b) \langle \proj{1}{2} \1_R, \proj{1}{2} \1_R \rangle} \,,\nonumber	
	 \end{align}
by choosing $b$ sufficiently small (i.e. $b< \bar b(1) \wedge\frac 1 {2 K_\beta}$).  Now recalling the definition of the error function $M(\beta)$ from~\eqref{e.M} this gives us  
\begin{align*}\label{}
\E_{\beta}^{Vil}\left[e^{i2\pi \langle m,\proj{1}{2} \1_R \rangle} \right]
& \leq e^{-\frac 1 {4 (2\pi)^2 \beta}  M(\beta) C(b)  \langle \proj{1}{2} \1_R, \proj{1}{2} \1_R \rangle}\,. 
\end{align*}
Using~\eqref{e.end} together with Proposition \ref{p.energy_wilson}, this leads us to
\begin{align*}
\E_{\beta}^{Vil}\left[e^{i\langle \d \theta, \1_R \rangle} \right] & \leq 
e^{-\frac 1 {2 \beta} \langle \proj{1}{2} \1_R, \proj{1}{2} \1_R \rangle  } e^{- \frac 1{2\beta}  \big( \frac {C(b)} {2 (2\pi)^2} \big) M(\beta) \langle \proj{1}{2} \1_R, \proj{1}{2} \1_R \rangle  }\,,
\end{align*}
which concludes our proof with $K:= \frac {C(b)} {2 (2\pi)^2}$.

In order to obtain a spin-wave improvement also in the case of the infinite-volume limit (whose existence is provided in Proposition \ref{pr.IVL}), notice as in \cite[Section 5]{Sourav},  that our above spin-wave improvement is uniform in $\Lambda_j \nearrow \Z^4$ which thus concludes the  proof of Theorem \ref{th.main}.

\end{proof}

\section{The $e^{-\pi^2 \beta}$ correction to the free energy}\label{s.FE}

\subsection{Variance of the Coulomb gas.}
The objective of the following section is to obtain quantitatives bounds on the variance of a Coulomb gas (Definition \ref{d.Coulomb})  which match with the RG predictions from \cite{guth1980}. 
This will be key to obtain bounds on the free-energy of the Coulomb gas. To improve our results, we will need a more accurate error function that will depend on the edge we are looking as in \cite{GS2}. 
\begin{align}
\tilde M (\beta,f):= \E^{Vil}_\beta \left[\var^{IG}\left (-\frac{\d\theta(f)}{2\pi}, (2\pi)^2\beta\right ) \right],
\end{align}
for any $f\in \overrightarrow C^2(\Lambda)$.
\begin{proposition}\label{p.lower_bound_variance_q}
Let $h$ be a real valued $3$-form and let $q$ be a Coulomb gas at inverse-temperature $\beta$ (with either boundary condition). Then,
\begin{align*}
\E^{Coul}_\beta\left[  \langle q,  h \rangle^2\right] \geq \langle \tilde M(\beta,\cdot )   \d^* h, \d^* h \rangle.
\end{align*}
\end{proposition}
The proof of this proposition follows in the same way as equation (6.2) of \cite{GS2}.
\begin{proof}
	We start using the coupling of Proposition \ref{p.decouplingF}. We may thus write $q:= \d m$ and we now use the law of the total variance to see that
	\begin{align*}
	\E^{Coul}_\beta\left[  \langle q,  h \rangle^2\right]&= \E^{Vil}_\beta\left[ \var\left[ \langle  \d m,  h \rangle \mid \theta \right] \right] + \var^{Vil}_\beta\left[\E\left[ \langle  \d m,  h \rangle \mid \theta \right]\right ]\\
	&\geq  \E^{Vil}_\beta\left[ \var\left[ \langle m, \d^* h \rangle \mid \theta \right] \right]\\
	&= \sum _{f \in C^{2}} \E_\beta^{Vil}\left[ \var^{IG}\left (-\frac{\d\theta(f)}{2\pi}, (2\pi)^2\beta\right ) \right] (\d^* h(f))^2\\
	&=\langle \tilde M(\beta,\cdot )   \d^* h, \d^* h \rangle,
	\end{align*}
	where in the third line we use that the law of $m$ conditionally on $\theta$ is that of an integer-valued Gaussian centered at $-\d\theta(f)/(2\pi)$ and at inverse-temperature $(2\pi)^2\beta$.
\end{proof}

\subsection{A sharper lower bound on $\tilde M(\beta,f)$.} It is clear that for any $f\in \overrightarrow{C}^2$, we have that $\tilde M(\beta,f)\geq(2\pi\beta)^{-1} M(\beta)$. The following proposition improves this bound.

\begin{proposition}\label{p.tilde_M}
	For any $\delta>0$, there exists $\beta_0<\infty$ such that the following holds: for any $\beta\geq \beta_0$,  there exists $L_0\in\N$ such that  for any $j\geq L_0$, if one considers the Villain coupling $(\theta,m)$ on $\Lambda_j$ with either free or Dirichlet boundary conditions, then for all $2$-cells $f$ at distance greater than or equal to $L_0$ of the boundary one has that	
	\begin{align*}
	\tilde M(\beta,f)\geq e^{-\pi^2 \beta (1+\delta) }.
	\end{align*}
	
	In simpler terms, for any $2$-cell $f$ of $\Z^4$ one has that
	\begin{align*}
	\liminf_{\beta\nearrow\infty} \liminf_{\Lambda_j\nearrow \Z^4} \frac{1}{\beta}\log (\tilde M(\beta,f))\geq -\pi^2
	\end{align*}
\end{proposition}

To prove the proposition we will need to control the probability that $\d \theta$ is big.
\begin{lemma}\label{l.lower_bound_dtheta}
	Let $(\theta,m)$ be a Villain model and $\varrho$ be a Gradient spin-wave both at the same inverse temperature $\beta$ and with the same boundary condition. We have that for any $f\in \overrightarrow{C}^2$, $a \in (0,2 \pi)$ and any $0\leq \epsilon \leq 2\pi -a $
	\begin{align*}
	\FK{\beta}{Vil}{\d \theta(f) \mod 2\pi  \in (a, a+2\epsilon)} \geq 
	2\, \FK{\beta}{GSW}{\varrho(f) \mod 2\pi \in(\pi-\epsilon, \pi )}.
	\end{align*}
	Here, for any $x\in \R$ we identify $x\mod 2\pi$ with a number in $[-\pi,\pi)$.
\end{lemma}
\begin{proof}
	We use again the coupling from Proposition \ref{p.decouplingF} to see that
	\begin{align*}
	\d \theta + 2\pi m = \varrho + (2\pi)\d^* \Delta^{-1} q,
	\end{align*}
	where $\varrho$ and $q$ are independent. Using this we see that
	\begin{align*}
	&\P^{Vil}_\beta(\d \theta(f) \mod 2\pi \in (a,a+2\epsilon)) \\
	&\hspace{0.1\textwidth} \dgreen{=} \P^{Vil}_\beta\left(\varrho(f) + (2\pi)\d^* \Delta^{-1} q(f) \mod 2\pi \in (a,a+2\epsilon) \right)\\
	&\hspace{0.1\textwidth}\geq \E^{Coul}_{\beta}\left[\P^{GSW}_\beta\left (\varrho(f) \mod 2\pi \in (a_q,b_q )\right ) \mid q \right], 
	\end{align*}
	where $[a_q,b_q] \subseteq [-\pi,\pi)$ is the interval (or union of two-intervals) of size $2\epsilon$ whose boundary are $a_q=a- (2\pi)\d^* \Delta^{-1} q \mod 2\pi$ and $b_q= a+2\epsilon - (2\pi)\d^* \Delta^{-1} q)$. Using that the interval $[a_q,b_q]$ has size $2\epsilon$, the fact that $\varrho(f)$ is a centred Gaussian random variable, we see that for any $q$
	\begin{align*}
	\FK{\beta}{GSW}{\varrho(f) \mod 2\pi \in (a_q,b_q )}\geq 2\FK{\beta}{GSW}{\varrho(f) \mod 2\pi \in(\pi-\epsilon, \pi)}\,,
	\end{align*}
	from where we conclude.
\end{proof}

We shall need to lower bound the probability that a Gaussian spin-wave is big in a given $2$-form. This is done in the following lemma.
\begin{lemma}
 Let us work in the context of Proposition \ref{p.tilde_M}. We have that for any $f \in \overrightarrow C^{2}(\Lambda_n)$ at distance bigger than $L_0$ from the boundary
 \begin{align}\label{e.var_rho}
 \var(\rho(f))=\frac{1}{2\beta} + o_{L_0}(1),
 \end{align}
 where $o_{L_0}(1)$ goes to $0$ as $L_0$ goes to infinity, uniformly on $n\geq L_0$ and $f$. Thus,
 \begin{align}\label{e.var_rho_consequence}
 \inf_{n\in \N}\inf_{\substack{ f \in \overrightarrow{C^2}(\Lambda_j)\\ d(f, \partial \Lambda_j)\geq L_0} }\FK{\beta}{GSW}{\varrho(f)\in (\pi-\epsilon,\pi) } \geq \frac{\epsilon \sqrt{\beta }}{\sqrt{2\pi}} e^{- \beta \pi^2(1+o_{L_0}(1)) }
 \end{align}
\end{lemma}
\begin{proof}
	It is clear that \eqref{e.var_rho} implies \eqref{e.var_rho_consequence}. Thus, we only prove \eqref{e.var_rho}. To do that, let $\phi^*$ be a GFF on the $3$-forms at inverse-temperature $\beta$ and independent of $\rho$, and $\rho^*=\d^* \phi^*$ (we define $\rho^*=\mathring \d^* \rho^*$ if we are working with $0$-boundary condition). We note that $\rho + \rho^*$ has the law of a white noise on the $2$-forms at inverse temperature $\beta$. Furthermore, as $\Lambda_j\nearrow \Z^4$, we have that both $\rho$ and $\rho^*$ converge in law \footnote{and in all their moments as they are Gaussian random variables} to $\rho_\infty$ and $\rho^*_\infty$.
In this case, we have that
	\begin{align*}
	\var(\rho_\infty(f))= \frac{1}{\beta}\langle \proj{1}{2} \1_f , \proj{1}{2} \1_f\rangle \text{ and } 
	\var(\rho^*_\infty(f))= \frac{1}{\beta}\langle \proj{3}{2} \1_f , \proj{3}{2} \1_f\rangle.
	\end{align*}
	Where the projections are defined in the $2$-forms of $\Z^4$ as in Remark \ref{r.Odecomp}. By using the symmetry of $\Z^4$ between the hyper-cubes and the vertices (or more precisely between itself and its Hodge dual) we have that
	\begin{align*}
	\var(\rho_\infty(f))= \var(\rho^*_\infty(f)).
	\end{align*}
	We conclude using that their sum is equal to $\frac{1}{\beta}$.	
\end{proof}

We can now prove Proposition \ref{p.tilde_M}.

\begin{proof}[Proof of Proposition \ref{p.tilde_M}]
	We note, using \eqref{e.variance_IG_bound},  that for any $\beta >10$
	\begin{align*}
 \tilde M(\beta,f)&= \E^{Vil}_\beta\left[ \var^{IG}\left (\frac{d\theta(f)}{2\pi},(2\pi)^2\beta \right )\right]\\
&\geq \frac{1}{16} e^{-2\pi \beta \epsilon } \P^{Vil}_\beta\left ( \d \theta(f) \mod 2\pi \in(-\pi, -\pi+\epsilon) \cup (\pi-\epsilon , \pi \right )),\\
&\geq \frac{1}{8} e^{-2\pi \beta \epsilon } \P^{GSW}_\beta\left ( \varrho(f) \in(\pi-\epsilon , \pi \right )),
\end{align*}
where in the last equation we used Lemma \ref{l.lower_bound_dtheta}. Using \eqref{e.var_rho_consequence} and taking $\epsilon= \beta^{-1}$, we have that
\begin{align*}
\tilde M(\beta,f) \geq \frac{e^{-2\pi}}{8\sqrt{\beta}} e^{-\beta \pi^2(1+o_{L_0}(1))}.
\end{align*}
We conclude from here by taking $\beta$ big enough and absorbing the constant in the $\delta>0$ term.
\end{proof}

\subsection{The free-energy of the Coulomb gas.} In this section, we obtain a lower bound for the partition function of a Coulomb gas. More precisely,
\begin{proposition}\label{p.free_enery_Coulomb}
	For any $\delta>0$, there exists $\beta_0<\infty$ such that for any $\beta\geq \beta_0$
	\begin{align*}
	f_j^{Coul}(\beta)= \frac{1}{4(2j)^4}\ln(Z^{Coul}_{\beta,\Lambda_j})&\geq \frac{3}{8}\int_{\beta}^{\infty} e^{-\pi^2\tilde \beta(1+\delta)} d\tilde \beta +O(j^{-1})\,,
	\end{align*}
	where $\Lambda_j:=[-j,j]^4\cap \Z^4$. This result is true for both free and zero boundary condition. Furthermore,
	\begin{align} \label{e.d_free_energy_Coulomb}
	 \frac{1}{4(2j^4)}\frac{d}{d\beta}f_j^{Coul}(\beta)\leq - \frac{3}{8} e^{-\beta \pi^2(1+\delta)}(1+O(j^{-1}))
	\end{align}
\end{proposition}

To prove Proposition \ref{p.free_enery_Coulomb}, it will be useful to recall that
\begin{align}\label{e.derivate_free_energy}
\frac{d}{d \beta }\ln(Z^{Coul}_{\beta,\Lambda_j}) = -\frac{1}{2}\E^{Coul}_{\beta}\left[\langle q, \Delta^{-1} q \rangle \right]. 
\end{align}
Having this in mind, it is clear that the following lemma is useful.
\begin{lemma}Let $q$ be a Coulomb gas at inverse-temperature $\beta$ and $\phi$ be a Gaussian free field on the $3$-forms at inverse-temperature $1$
	\begin{equation}\label{e.lb_Coulomb}
	\E^{Coul}_{\beta}\left[\langle q, (-\Delta^{-1}) q \rangle \right] \geq e^{-\beta \pi^2(1+\delta)}\E^{GFF}_{1}\left[\langle \d^* \phi , \d^* \phi \rangle \right] (1-O(j^{-1})).
	\end{equation}
\end{lemma}

\begin{proof}
	Note that
	\begin{align}\label{e.energy_q}
	\langle q, (-\Delta^{-1}) q \rangle  = \langle \d^* (\Delta^{-1}) q, \d^* (\Delta^{-1}q) = \sum_{f \in C^2(\Lambda_j)} \langle \d^*(\Delta^{-1}) q, f \rangle^2.
	\end{align}
	
	We will now show that there exists \avelio{$\tilde L_0\geq L_0$ (where $L_0$ is defined in Proposition \ref{p.tilde_M})}  such that for any $f\in C^2(\Lambda_n)$ \avelio{whose support is at distance at} least $\tilde L_0$ from the boundary
	\begin{align}\label{e.want_to_prove_variance_q}
	\E\left[\langle \d^*(\Delta^{-1}) q, 1_f \rangle^2 \right]  &\geq e^{-\beta\pi^2(1+\delta) }\langle \proj{3}{2} \1_f, \proj{3}{2} \1_f \rangle \\
	&=  e^{-\beta\pi^2(1+\delta) } \E^{GFF}_{1}\left[\langle \d^* \phi  , 1_f \rangle^2 \right]. 	
	\end{align}
	To prove \eqref{e.want_to_prove_variance_q}, we use Proposition \ref{p.lower_bound_variance_q} to see that
	\begin{align*}
	\E\left[\langle \d^*(\Delta^{-1}) q, \1_f \rangle^2 \right]  &\geq \langle \tilde M (\beta,\cdot)\d^*\d(\Delta^{-1}) \1_f, \d^* \d (\Delta^{-1}) \1_f \rangle\\
	&\geq e^{-\beta \pi^2(1+\delta)}\langle \1_{d(\cdot,\partial \Lambda_j)\geq L_0} \proj{3}{2} \1_f, \proj{3}{2} \1_f \rangle.
	\end{align*}
	We just need to show that for all $f$ that are at distance greater than \avelio{$\tilde L_0$} from the boundary 
	\begin{align}\label{e.1/2}
	\langle \1_{d(\cdot,\partial\Lambda_j)\geq L_0} \proj{3}{2} \1_f, \proj{3}{2} \1_f \rangle \geq \frac{1}{2}\langle \proj{3}{2} \1_f, \proj{3}{2} \1_f \rangle.
	\end{align}
	This follows from the fact that due to the explicit characteristic of $\Delta^{-1}$ we have that 
	\begin{align*}
	\d^*\Delta^{-1} \d f(f_1,f_2)\leq  \frac{C}{\|f_1-f_2\|^4}+ O(\|f_1-f_2\|^5).
	\end{align*}
	
	Now, we comeback to \ref{e.energy_q} and note that
	\begin{align*}
	\E\left[\langle q, (-\Delta)^{-1}q \rangle \right] &\geq e^{-\beta \pi^2(1+\delta)} \sum_{\substack{\substack{f \in C^{2}(\Lambda_j)\\ d(f,\partial \Lambda_j) > \tilde L_0}}} \E^{GFF}_1(\langle \d^* \phi, \1_f \rangle)\\
	&= e^{-\beta \pi^2(1+\delta)}\bigg ( \E^{GFF}_1\left[\langle \d^*\phi, \d^* \phi \rangle \right]-   \sum_{\substack{f \in C^{2}(\Lambda_j)\\ d(f,\partial \Lambda_j)\leq \tilde L_0}}\E^{GFF}_1\left[\langle\d^* \phi, \1_f \rangle ^2 \right]\bigg ).
	\end{align*}
	To control the term with the sum in the last equation, we note that both for free and $0$ boundary condition, we have that by Proposition \ref{p.basic_calculus} and Remark \ref{r.counting_dimensions}
	\begin{align*}
	\E_1^{GFF}[\langle \d^* \phi, \d^* \phi \rangle]= 3 (2j)^4 + o(j^3).
	\end{align*}
	Furthermore
	\begin{align*}
	\sum_{\substack{f \in C^{2}(\Lambda_j)\\ d(f,\partial \Lambda_j)\leq \tilde L_0}}\E^{GFF}_1\left[\langle\d^* \phi, \1_f \rangle ^2 \right]& = \sum_{\substack{f \in C^{2}(\Lambda_j)\\ d(f,\partial \Lambda_j)\leq \tilde L_0}} \langle \proj{3}{2}\1_f, \1_f\rangle\\
	&\leq \sum_{\substack{f \in C^{2}(\Lambda_j)\\ d(f,\partial \Lambda_j)\leq  \tilde L_0}} 1 \\
	&\leq 12  \tilde L_0 j^3,
	\end{align*}
	from where we finally conclude. 
\end{proof}

Finally, we prove Proposition \ref{p.free_enery_Coulomb}.
\begin{proof}[Proof of Proposition \ref{p.free_enery_Coulomb}]
We start by using \eqref{e.derivate_free_energy} and \eqref{e.lb_Coulomb} to see that	\begin{align*}
	\frac{d}{d\beta}\ln(Z^{Coul}_{\Lambda_n})\leq - \frac{1}{2}e^{-\pi^2(1+\delta)}\E^{GFF}_{1}\left[\langle \d^* \phi , \d^* \phi \rangle \right] (1-o(n^{-1})).
	\end{align*}
	Recalling that 
	\begin{align*}
	\E^{GFF}_{1}\left[\langle \d^* \phi , \d^* \phi \rangle \right]= 3(2n)^4 + O(n^3),
	\end{align*}
	and using the fact that $f_n^{Coul}(\infty)=0$, we conclude.
\end{proof}

\subsection{Free-energy of Villain model.}
We now prove Theorem \ref{t.d_free_energy}.
\begin{proof}[Proof of Theorem \ref{t.d_free_energy}]We use Corollary \ref{l.decoupling_partition_function}, to see that
	\begin{align*}
	\frac{d}{d \beta} \log(Z^{Vil}_{\beta, \Lambda_j})&= \frac{d}{d \beta}(\log( Z^{GSW}_{\beta,\Lambda_j}) + \log(Z^{Coul}_{\beta,\Lambda_j}))\\
	&= -\frac{1}{2\beta} (3(2j)^4+O(j^3))+\frac{d}{d\beta}\log(Z^{Coul}_{\beta,\Lambda_j})
	\end{align*}
	We conclude using \eqref{e.d_free_energy_Coulomb}.
\end{proof}

\section{Concluding remarks}

We conclude with two remarks. 

\begin{remark}\label{r.SouravUB}
In \cite[Lemma 7.12]{Sourav}, Chatterjee obtains an upper bound which happens to be asymptotically sharp for some natural regime in $(\beta,\gamma)$ as follows: one first conditions on the values of the $1$-form $\theta$ on all edges $e$ except the ones along $\gamma$ and one considers the worse-case scenario for $\Eb{W_\gamma \md \{\theta_e\}_{e \notin \gamma}}$. As pointed out to us by Malin Palö Forsström, this technique also works in the case where $G=U(1)$. Analyzing the corresponding worse case scenario, this gives the following upper bound for Wilson loops
\begin{align*}\label{}
|\EFK{\Lambda,\beta}{}{W_\gamma}| \leq \exp\left (-\left (\frac{1} {12 \beta} +o(1/\beta)\right )|\gamma|\right )\,.
\end{align*}
One can check that $\tfrac 1 {12}< C_{GFF}$ (which was defined in~\eqref{e.cGFF}). Interestingly this shows that this natural technique captures (for some suitable regime of $\beta$ and $|\gamma|$) the correct asymptotics for $\<{W_\gamma}$ when the gauge group is discrete (see also \cite{malin,cao}) while it never provides the correct decay for the continuous Gauge group $U(1)$. 
\end{remark}

\begin{remark}\label{}
The phase transition for $U(1)$ lattice gauge theory between a deconfining and a confining phase in \cite{FSrestoration} shares some similarities with the BKT transition proved in \cite{FS}. Fröhlich-Spencer conjectured in \cite{FS1983} that the $2d$ Villain model should behave at large scale like $e^{ i \beta_{eff}^{-1} GFF}$. The present work suggests a similar conjecture for $U(1)$ lattice gauge theory with Villain interaction. As we have seen in Section \ref{s.decoupling}, one should  state the corresponding conjecture for the $2$-form $\d \theta$ rather than for the $1$-form $\theta$. We then expect that  $\d \theta$  should behave at large scales like $e^{i \beta_{eff}^{-1} \varrho}$ where $\varrho$ is the gradient spin-wave (at $\beta\equiv 1$) defined in Definition \ref{d.SW}. As discussed in Subsection \ref{ss.links}, note that this conjecture is proved for smooth enough observables (which do not include Wilson observables) in \cite{driver1987convergence}. Inspired by our earlier work \cite{GS1} and thanks to the analogy between the BKT and the confining/deconfining transition, we should then expect a certain phase transition to happen for the statistical reconstruction of the gradient spin-wave $\varrho$ given $e^{i \beta_{eff}^{-1} \varrho}$. This will be the subject of the work in progress \cite{GS4}.
\end{remark}

%


\bibliographystyle{alpha}
\bibliography{biblio-SWYM}

\end{document}